\documentclass{amsart}
\usepackage{amsmath}
\usepackage{amssymb}
\usepackage{amsthm}
\usepackage{enumerate}
\usepackage{enumitem}
\usepackage[pdftex]{graphicx}
\usepackage{caption}
\theoremstyle{definition}
\newtheorem{definition}{Definition}[section]
\theoremstyle{plain}
\newtheorem{lemma}[definition]{Lemma}
\newtheorem{theorem}[definition]{Theorem}

\newtheorem{proposition}[definition]{Proposition}
\newtheorem{corollary}[definition]{Corollary}
\theoremstyle{remark}

\makeatletter
\@namedef{subjclassname@2020}{
  \textup{2020} Mathematics Subject Classification}
\makeatother

\newcommand{\colorgL}{\mathcal L^n_{\textrm{cg}}}
\newcommand{\colorgT}{T^n_{\text{cg}}}
\newcommand{\colorgK}{{\mathbf K}^n_{\textrm{cg}}}
\newcommand{\randgT}{T^n_{\textrm{gcg}}}
\newcommand{\randgF}{\mathbb G^n_{\textrm{gcg}}}
\newcommand{\frcolgT}{T^n_{\overline{m}\text{-fcg}}}
\newcommand{\frcolgK}{\mathbf{K}^n_{\overline{m}\text{-fcg}}}
\newcommand{\frcolgTc}{T^{c,n}_{\overline{m}\text{-fcg}}}
\newcommand{\randfrcolgT}{T^n_{\overline{m}\text{-gfcg}}}

\newcommand{\myth}{\operatorname{Th}}
\newcommand{\mytp}{\operatorname{tp}}
\newcommand{\mytptwo}{\operatorname{TP}_2}
\newcommand{\myntptwo}{\operatorname{NTP}_2}
\newcommand{\mynsop}{\operatorname{NSOP}}
\newcommand{\myacl}{\operatorname{acl}}
\newcommand{\mydcleq}{\operatorname{dcl}^{\textit{eq}}}
\newcommand{\myacleq}{\operatorname{acl}^{\textit{eq}}}
\newcommand{\mydiag}{\operatorname{Diag}}
\newcommand{\mysop}{\operatorname{SOP}}
\newcommand{\mycol}{\operatorname{col}}
\newcommand{\compgraph}{K_{\overline{m}}}
\newcommand{\subcompgraph}{{\widetilde K}_{\overline{m}}}

\def\mathpalette#1#2{%
	\mathchoice
	{#1\displaystyle{#2}}%
	{#1\textstyle{#2}}%
	{#1\scriptstyle{#2}}%
	{#1\scriptscriptstyle{#2}}}

\def\Ind#1#2{#1\setbox0=\hbox{$#1x$}\kern\wd0\hbox to 0pt{\hss$#1\mid$\hss}
	\lower.9\ht0\hbox to 0pt{\hss$#1\smile$\hss}\kern\wd0}

\def\ind{\mathop{\mathpalette\Ind{}}}

\def\notind#1#2{#1\setbox0=\hbox{$#1x$}\kern\wd0
	\hbox to 0pt{\mathchardef\nn=12854\hss$#1\nn$\kern1.4\wd0\hss}
	\hbox to 0pt{\hss$#1\mid$\hss}\lower.9\ht0 \hbox to 0pt{\hss$#1\smile$\hss}\kern\wd0}

\def\nind{\mathop{\mathpalette\notind{}}}

\newcommand{\myindzero}{\ind^\text{n}}

\newcommand{\myindalg}{\ind^\text{alg}}
\newcommand{\mynindalg}{\nind^\text{alg}}
\newcommand{\myinddiv}{\ind^\text{d}}
\newcommand{\myninddiv}{\nind^\text{d}}
\newcommand{\myindfork}{\ind^\text{f}}

\begin{document}
\title[A note on generic $n$-partite graphs]{A note on generic $n$-partite graphs}
\author[M. Fujita]{Masato Fujita}
\address{Department of Liberal Arts,
Japan Coast Guard Academy,
5-1 Wakaba-cho, Kure, Hiroshima 737-8512, Japan}
\email{fujita.masato.p34@kyoto-u.jp}

\begin{abstract}
	An $n$-partite graph is a graph such that every vertex has a color in $\{1,\ldots,n\}$ and every two vertices of the same color are not adjacent.
	We study the model comparisons of the theories of $n$-partite graph and $\compgraph$-free $n$-partite graph, where $\compgraph$ is a complete graph of a given size $\overline{m}$.
	The model companion of the theory of $n$-partite graph is simple and has IP.
	The model companion of the theory of $\compgraph$-free $n$-partite graph has $\mytptwo$, $\mysop_3$ and $\mynsop_4$ if $n > 2$.
	Forking independence coincides with dividing independence in this theory.
\end{abstract}

\subjclass[2020]{Primary 03C45; Secondary 05C15}

\keywords{Color; Generic $n$-partite graphs, $\operatorname{NSOP}_4$}

\thanks{The author was supported by JSPS KAKENHI Grant Number JP25K07109.}

\maketitle

\section{introduction}\label{sec:intro}
An $n$-partite graph is a graph such that every vertex has a color in $\{1,\ldots,n\}$ and every two vertices of the same color are not adjacent.
In this paper, we study the model companions of the theories of $n$-partite graph and $\compgraph$-free $n$-partite graph, where $\compgraph$ is a complete graph of a given size $\overline{m}$.
See \cite[Definition 3.2.8]{TZ} for the definition of model companion.

We recall previous works related to this study.
First, we recall studies on non-colored graphs.
It is well known that the theory of random graph, which is the model companion of the theory of non-colored graph, is simple.
See \cite[Corollary 7.3.14]{TZ} for instance.
Let $K_m$ be the complete graph of $m$ vertices.
A complete graph is a graph such that every two vertices are adjacent to each other.
The theory of generic $K_m$-free graph has $\mytptwo$ by \cite[Example 3.13]{Chernikov} and has $\mynsop_4$ by \cite[Theorem 1.1]{Conant}.
In particular, they are not simple.
In addition, the notions of dividing and forking do not coincide in the formulas in the theory of generic $K_m$-free graph by \cite[Theorem 6.2]{Conant2}, but forking independence coincides with dividing independence \cite[Theorem 5.3]{Conant2}.

The model companion of the theory of $\compgraph$-free bipartite graph, which is colored by two colors, is studied in \cite{CK}.
%Observe that a bipartite graph is a $2$-colored graph.
Conant and Kruckman show that this theory is $\mynsop_1$ (\cite[Theorem 4.11]{CK}) and forking and dividing independence coincide with each other in this theory \cite[Corollary 4.24]{CK}, though the notions of dividing and forking do not coincide in formulas \cite[Proposition 4.25]{CK}.
Recall that $\mynsop_1$ is equivalent to $\mynsop_2$ by \cite{M}.

Suppose $n \geq 2$.
Let $R$ be a binary predicate and $C_1, \ldots, C_n$ be unary predicates.
Put $\colorgL:=(R,C_1,\ldots, C_n)$.
Let $\colorgT$ be the $\colorgL$-theory defined as follows:
\begin{enumerate}
	\item[(1)] $\forall x\ \neg R(x,x)$
	\item[(2)] $\forall x,y\ R(x,y) \leftrightarrow R(y,x)$
	\item[(3)] $\forall x \ \bigvee_{i=1}^n C_i(x)$
	\item[(4)] $\forall x \ \bigwedge_{i=1}^n (C_i(x) \rightarrow \bigwedge_{j \neq i}\neg C_j(x))$
	\item[(5)] $\forall x,y \ R(x,y) \rightarrow \bigwedge_{i=1}^n (C_i(x) \rightarrow \neg C_i(y))$
\end{enumerate}
The language $\colorgL$ and the theory $\colorgT$ are called the language and the theory of \textit{$n$-partite graph} ($n$-colored graph), respectively.
The subscript `cg' in $\colorgL$ and $\colorgT$ is an abbreviation of colored graph.

In this paper, we study extensions of $\colorgT$.
We summarize the terminology and notations used in this paper.
A \textit{graph} means an $\colorgL$-structure.
For a graph $G$, a \textit{copy} of $G$ is a graph isomorphic to $G$.
Elements in a graph are called \textit{vertices}.
We say that a vertex $v$ is of \textit{color} $k$ if $C_k(v)$ holds.
For a vertex $v$, the color of $v$ is denoted by $\mycol(v)$.
Let $\mycol(G)$ denote the set of colors of vertices in a graph $G$.
A vertex $x$ is \textit{adjacent} to another vertex $y$ if $R(x,y)$ holds.
We say that $x$ is separated from $y$ if $\neg R(x,y)$ holds.
For two graphs $G_1$ and $G_2$, we say that $G_1$ is \textit{fully adjacent} to $G_2$ if, for every $1 \leq c \leq n$ and $v \in G_1$ of color $c$, $v$ is adjacent to all the vertices in $G_2$ of color $\neq c$.
$G_1$ is \textit{fully separated} from $G_2$ if every $v \in G_1$ is separated from every $v_2 \in G_2$.
We say that a graph $G$ is \textit{monochromatic} if there exists $1 \leq c \leq n$ such that $C_c(x)$ holds for every $x \in G$.

In Section \ref{sec:colored}, we study the model companion of $\colorgT$.
The following theorem summarizes the major results in Section \ref{sec:colored}.
\begin{theorem}\label{thm:main1}
	The model companion $\randgT$ of $\colorgT$ is complete, simple and has IP.
\end{theorem}

Let $\overline{m}=(m_1,\ldots,m_n)$ be a sequence of positive integers.
Let $\compgraph$ be the graph consisting of $m_i$ vertices of color $i$ for $1 \leq i \leq n$ and every two vertices with different colors are adjacent.
Let $\frcolgT$ be the theory of \textit{$\compgraph$-free} $n$-partite graph, that is, every model of $\frcolgT$ does not have a copy of $\compgraph$ as a subgraph.

Section \ref{sec:free} is devoted for a study of the model companion of $\frcolgT$.
The following theorem summarizes the major results in Section \ref{sec:free}.
\begin{theorem}\label{thm:main2}
	Suppose $n>2$. The model companion $\randfrcolgT$ of $\frcolgT$ is complete, and has $\mytptwo$, $\mysop_3$ and $\mynsop_4$.
	Forking independence coincides with dividing independence in this theory.
\end{theorem}
Let us compare the above theorem with the results in \cite{CK}, which examines $\randfrcolgT$ for $n=2$.
The theory $\randfrcolgT$ for $n>2$ shares several common properties with the case of $n=2$.
For instance, they are $\mytptwo$, and forking independence coincides with dividing independence in these theories.
However, $\randfrcolgT$ has $\mynsop_2$ for $n=2$, but has $\mysop_3$ for $n>2$.

We employ standard conventions used in the papers of model theory.
We represent a structure and its universe by the same symbol such as $M$ and $\mathbb M$.
We use a small alphabet such as $a$ and $x$ to represent a single element of a universe or a single variable.
We use an overlined small alphabet such as $\overline{a}$ and $\overline{x}$ to represent a tuple of elements of a universe or a tuple of variables.
In this note, a tuple means a sequence of elements of finite length not containing two identical elements.
A tuple $\overline{a}$ of finite length is sometimes identified with the sets of elements in $\overline{a}$.
We write $AB$ instead of $A \cup B$ for sets $A$ and $B$. 
The cardinality of a set $S$ and the length of a tuple $\overline{a}$ are denoted by $|S|$ and $|\overline{a}|$, respectively.
Let $\overline{x}=(x_1,\ldots,x_n)$ is a tuple of $n$-variables.
We denote $\exists x_1,\ldots,x_n$ by $\exists \overline{x}$.
We define $\forall \overline{x}$ in the same manner.
For a unary predicate $P$, $P(\overline{x})$ means $\bigwedge_{i=1}^n P(x_i)$.
The algebraic closure of a small set $A$ in a monster model is denoted by $\myacl(A)$.

For a graph $G$, we denote the set of vertices in $G$ by the same symbol $G$.
For a tuple of $\overline{a}$ of vertices in a graph, the set of vertices in $\overline{a}$ and the graph consisting of $\overline{a}$ are denoted by the same symbol $\overline{a}$.
For two tuples $\overline{a}$ and $\overline{b}$, the concatenation of $\overline{a}$ with $\overline{b}$ is denoted by $\overline{a}\overline{b}$ or $\overline{a}{}^{\frown}\overline{b}$.

\section{Generic $n$-partite graphs}\label{sec:colored}

We study the theory of generic $n$-partite graphs $\randgT$ in this section.
%Differently from the theory of random colored trees, $\randgT$ has IP.
%Note that each model of $\randgT$ is a $K_{n+1}$-free graph if we forget the colors, but we show that $\randgT$ is simple.
%$\randgT$ is rather similar to the theory of random graphs, which has IP and is simple.
We recall a basic definition.
\begin{definition}
	Let $\mathcal L$ be a language.
	Let $\mathbf K$ be a family of finitely generated $\mathcal L$-structures with $\emptyset \in \mathbf K$ which is closed under $\mathcal L$-isomorphisms.
	The following properties are well-known.
	\begin{enumerate}
		\item[(a)] (Heredity) For every $ A \in \mathbf K$, every finite substructure of $ A$ is a member of $\mathbf K$.
		\item[(b)] (Amalgamation) Let $ A,  B_1,  B_2 \in \mathbf K$ and $f_i: A \to  B_i$ be embeddings for $i=1,2$. There exist $ C \in \mathbf K$ and embeddings $g_i: B_i \to  C$ for $i=1,2$ such that $g_1 \circ f_1=g_2 \circ f_2$.
	\end{enumerate}
	
	Suppose $\mathcal L$ is a relational language.
	We say that $\mathbf K$ possesses the \textit{free amalgamation property} if, for all $A_1, A_2, C \in \mathbf K$ with $C \subseteq A_1 \cap A_2$, there exist $D \in \mathbf K$ and $\mathcal L$-embeddings $f_i:A_i \to D$ such that $f_1(C)=f_2(C)$, $D=f_1(A_1) \cup f_2(A_2)$ and, for every symbol $R \in \mathcal L$ and every $\overline{a} \in (A_1A_2C)^n$, $\overline{a} \in (A_iC)^n$ for some $i=1,2$ if $D \models R(\overline{a})$ under the identification of $A_i$ with $f_i(A_i)$.
\end{definition}

Let $\randgT$ be the extension of $\colorgT$ by the following sentences (6):
Let $1 \leq k \leq n$ and $m_1,\ldots, m_n$ be nonnegative integers.
Let $\sigma: \bigcup_{i=1}^n (\{i\} \times \{1,\ldots,m_i\}) \to \{0,1\}$ be an arbitrary map.
We put $R^0(x,y)=\neg R(x,y)$ and $R^1(x,y)=R(x,y)$.
\begin{enumerate}[start=6]
	\item $\displaystyle\forall \overline{x^1}, \ldots, \forall \overline{x^n}\ \left(\bigwedge_{i=1}^n C_i(\overline{x^i}) \rightarrow \exists y\ \bigwedge_{i \neq k}\bigwedge_{j=1}^{m_i}R^{\sigma(i,j)}(x_j^i,y) \wedge C_k(y)\right),$\label{cond:random}
	where $\overline{x^i}=(x_1^i,\ldots x^i_{m_i})$ for $1 \leq i \leq n$.
\end{enumerate}
The theory $\randgT$ is called the theory of \textit{generic $n$-partite graphs}.

Let $\colorgK$ be the collection of finite models of $\colorgT$ under the identification of structures isomorphic to each other.
For every $B_1, B_2 \in \colorgK$ having a common substructure $A \in \colorgK$, the \textit{free amalgam} of $B_1$ and $B_2$ over $ A$ is the model $ M$ of $\colorgT$ such that there exist $\colorgL$-embeddings $f_i: B_i \to M$ such that $f_1(B_1) \cap f_2(B_2)=f_1(A)=f_2(A)$, $f_1(B_1) \cup f_2(B_2)=M$, and $f_1(B_1) \setminus f_1(A)$ is fully separated from $f_2(B_2) \setminus f_2(A)$. 
It is easy to see that the free amalgam exists uniquely up to isomorphism.
We denote the free amalgam of $B_1$ and $B_2$ over $A$ by $B_1 \otimes_{A} B_2$.

For convenience, we assume that the empty structure $\emptyset$ is also a member of $\colorgK$.
The following lemma is obvious.
\begin{lemma}\label{lem:pre_amalgamation}
	$\colorgK$ enjoys the heredity property and the free amalgamation property.
\end{lemma}

By \cite[Theorem 4.4.4]{TZ}, $\colorgK$ has the Fra\"iss\'e limit $\randgF$.
$\myth(\randgF)$ is $\omega$-categorical and admits elimination of quantifiers by \cite[Theorem 4.4.7]{TZ}.

\begin{theorem}\label{thm:equaliyt_theory}
	The equality $\randgT \equiv \myth(\randgF)$ holds.
	In particular, $\randgT$ is complete.
\end{theorem}
\begin{proof}
	First, we show that $\randgF$ is a model of $\randgT$.
	Let $1 \leq k \leq n$ and $m_1,\ldots, m_n$ be nonnegative integers.
	Let $\sigma: \bigcup_{i=1}^n \{i\} \times \{1,\ldots,m_i\} \to \{0,1\}$ be an arbitrary map.
	For $1 \leq i \leq n$, $X_i$ be an arbitrary monochromatic subset of $\randgF$ of cardinality $m_i$ whose vertices are of color $i$.
	We enumerate $X_i$ as $(x_1^i, \ldots, x_{m_i}^i)$.
	We want to find $y \in G$ such that $\randgF \models \bigwedge_{i \neq k}\bigwedge_{j=1}^{m_i}R^{\sigma(i,j)}(x_j^i,y) \wedge C_k(y)$.
	Since $\randgF$ is the Fra\"iss\'e limit of $\colorgK$, there exists $H \in \colorgK$ such that $H$ is isomorphic to $X:={\bigcup_{i=1}^n X_i}$ as $\colorgL$-structures.
	Let $f$ be the isomorphism between them.
	We identify $H$ with $X$ under the isomorphism $f$.
	Let $H'$ be the $n$-partite graph expanded from $H$ by a single vertex $y'$ as follows:
	$x^i_j$ is adjacent to $y'$ if and only if $i \neq k$ and $\sigma(i,j)=1$.
	We paint the node $y'$ by the color $k$.
	$H'$ is obviously a member of $\colorgK$.
	Since the Fra\"iss\'e limit $\randgF$ is $\colorgK$-saturated in the sense of \cite[Definition 4.4.1]{TZ}, there exists an embedding $g:H' \to \randgF$ extending $f:H \to X$.
	Put $y=g(y') \in \randgF$, then $\randgF \models \bigwedge_{i \neq k}\bigwedge_{j=1}^{m_i}R^{\sigma(i,j)}(x_j^i,y) \wedge C_k(y)$.
	This means that $\randgF$ is a model of $\randgT$.
	
	Next, we show that every model of $\randgT$ is a model of $\myth(\randgF)$.
	Let $M$ be an arbitrary model of $\randgT$.
	Our task is to show that $\myth(M)=\myth(\randgF)$.
	By the L\"owenheim-Skolem theorem \cite[Corollary 2.3.2]{TZ}, there exists a countable model of $\myth(M)$.
	We may assume that $M$ is countable by replacing $M$ with a countable model of $\myth(M)$.
	
	Let $A, B \in \colorgK$ with $A \subseteq B$.
	Let $f:A \to M$ be an embedding.
	Since $B$ is a finite structure, by an easy induction on $|B \setminus A|$, we can easily extend $f$ to an embedding $g:B \to M$ using property (\ref{cond:random}).
	This means that $M$ is $\colorgK$-saturated.
	By \cite[Theorem 4.4.2]{TZ}, $\myth(\randgF)$ is $\omega$-categorical.
	Therefore, $M$ is isomorphic to $\randgF$.
	This implies that $\myth(M) =\myth(\randgF)$.
\end{proof}

We introduce some features of $\randgT$.

\begin{proposition}\label{prop:model_compansion}
	$\randgT$ is a model companion of $\colorgT$.
\end{proposition}
\begin{proof}
	Recall that $\myth(\randgF)$ admits quantifier elimination.
	By Theorem \ref{thm:equaliyt_theory}, $\randgT$ admits quantifier elimination.
	A nontrivial part is that every model $M$ of $\colorgT$ is a substructure of some model $N$ of $\randgT$.
	Put $M_0=M$.
	We define $M_i$ for $i<\omega$ as follows:
	
	For $1 \leq c \leq n$, let $\mathcal S_{i,c}$ the collection of finite sets $A$ of vertices in $M_{i-1}$ of color $\neq c$.
	We define $M_i$ as an extension of $M_{i-1}$ by adding a new vertex $p_{c,A}$ of color $c$ for every $1 \leq c \leq n$ and $A \in \mathcal S_{i,c}$ and joining $p_{c,A}$ with every vertex in $A$ by edges.
	It is obvious that $N=\bigcup_{i=1}^\infty M_i$ is a model of $\randgT$.
\end{proof}

\begin{definition}[{\cite{Simon_NIP}}]
	A theory $T$ has \textit{IP} if there exist a formula $\phi(\overline{x},\overline{y})$ and a subset $A$ of tuples from a model such that, for an arbitrary finite subset $B$ of $A$ and an arbitrary subset $I$ of $B$, $ \exists \overline{x}\ \bigwedge_{\overline{a} \in I} \phi(\overline{x},\overline{a}) \wedge \bigwedge_{\overline{a} \in B \setminus I} \neg \phi(\overline{x},\overline{a})$ holds.
	We say that $A$ is \textit{shuttered} by the formula $\phi(\overline{x},\overline{y})$ if the above condition is satisfied.
	We say that $T$ has \textit{NIP} if it does not have IP. 
\end{definition}

\begin{proposition}\label{prop:ip}
	$\randgT$ has IP.
\end{proposition}
\begin{proof}
	Let $M$ be a model of $\randgT$.
	Let $A$ be the subgraph of $M$ consisting of vertices whose color is not $1$.
	Let $B$ be an arbitrary finite subset of $A$ and $I$ be an arbitrary subset of $B$.
	By (\ref{cond:random}), we can find a vertex $y_I$ whose color is $1$ such that $\mathcal M \models R(x,y_I) \Leftrightarrow x \in I$ for $x \in B$.
	This means $A$ is shuttered by the formula $R(\cdot,\cdot)$.
	Therefore, $\randgT$ has IP.
\end{proof}

\begin{definition}
	A formula $\phi(\overline{x},\overline{y})$ has the \textit{tree property} with respect to $k$ if there is a tree of parameters $(\overline{a}_s\;|\; \emptyset \neq s \in \omega^{<\omega})$ such that $(\phi(\overline{x},\overline{a}_{s{}^\frown i})\;|\; i<\omega)$ is $k$-inconsistent for every $s \in \omega^{<\omega}$, and $\{\phi(\overline{x},\overline{a}_s)\;|\; \emptyset \neq s \subseteq \sigma\}$ is consistent for every $\sigma \in \omega^{\omega}$. 
	A theory $T$ is \textit{simple} if there is no formula with the tree property.
\end{definition}

\begin{proposition}\label{prop:simple}
	$\randgT$ is simple.
\end{proposition}
\begin{proof}
	Let $\mathbb M$ be a monster model of $\randgT$.
	Using the Kim-Pillay theorem \cite[Theorem 7.3.13]{TZ}, we show that $\randgT$ is simple.
	
	Let $A,B,C$ be a small subset of $\mathbb M$.
	We define the naive independence $A\myindzero_CB$ by $A \cap B \subseteq C$.
	We have only to show that $A\myindzero_CB$ satisfies conditions (a) through (f) in \cite[Theorem 7.3.13]{TZ}.
	
	Condition (a) called monotonicity and transitivity says that \[
	\overline{a}\myindzero_ABC
	 \Leftrightarrow  \overline{a}\myindzero_AB\text{ and } \overline{a}\myindzero_{AB}C.\]
	
	Condition (b) called symmetry says that ${\overline{a}}\myindzero_{A}{\overline{b}} \Leftrightarrow {\overline{b}}\myindzero_{A}{\overline{a}}$.
	
	Condition (c) called finite character says that ${\overline{a}}\myindzero_{A}{B}$ holds if ${\overline{a}}\myindzero_{A}{\overline{b}}$ for every tuple $\overline{b}$ from $B$ of finite length.
	
	These conditions are obviously satisfied.
	
	Condition (d) called local character says that there exists a cardinal $\kappa$, such that for all $\overline{a}$ and $B$, there exists $B_0 \subseteq B$ of cardinality less than $\kappa$ such that ${\overline{a}}\myindzero_{B_0}{B}$.
	We put $B_0 = \overline{a} \cap B$, then the relation $\myindzero$ satisfies local character with $\kappa=\omega$.
	
	Condition (e) called existence says that, for all $\overline{a}, A, C$, there exists $\overline{a'}$ such that $\mytp(\overline{a'}/A)=\mytp(\overline{a}/A)$ and ${\overline{a'}}\myindzero_{A}{C}$.
	Enumerate $\overline{a}=(a_1,\ldots,a_l)$. 
	By permuting the coordinate, we may assume that there exists $0 \leq m \leq l$ such that $a_i \in A$ if and only if $1 \leq i \leq m$.
	%Put $a'_i:=a_i$ for $1 \leq i \leq m$.
	Let $1 \leq k(i) \leq n$ be the color of $a_i$ for $m<i \leq l$.
	%For every finite subset $A_0$ of $A$ and $m<i \leq n$, we define $C(A_0):=\{b \in A_0\;|\; \mathbb M \models R(b,a_i)\}$ and $N(A_0):=A_0 \setminus C(A_0)$.
	
	Fix finite subsets $A_0 \subseteq A$ and $C_0 \subseteq C$.
	Put $c_i:=a_i$ for $1 \leq i \leq m$, and find $c_i$ for $m<i \leq l$ so that $\mathbb M \models C_{k(i)}(c_i)$, and $\mathbb M \models R(x,c_i)$ if and only if $\mathbb M \models R(x,a_i)$ for every $x \in A_0 \cup \{c_j\;|\;j<i\}$ by induction on $i$.
	It is possible by condition (\ref{cond:random}).
	Let $A'_0=A_0c_1\ldots c_l$.
	Consider the free amalgam $A'_0 \otimes_{A_0} A_0C_0$.
	There exists an embedding $A'_0 \otimes_{A_0} A_0C_0 \hookrightarrow \mathbb M$ fixing $A_0C_0$ by the axiom of $\randgT$.
	Therefore, the set consisting of atomic $\colorgL(A)$-formulas, their negations satisfied by $\overline{a}$ and the formulas of the form $x_i \neq c$ for $c \in C$ and $m<i \leq l$ is a partial type $q(\overline{x})$.
		
	Let $\overline{a'}=(a'_1,\ldots,a'_l)$ be a realization of $q(\overline{x})$ in $\mathbb M$.
	We have $\overline{a'} \cap C \subseteq A$. 
	For every quantifier-free $\colorgL(A)$-formula $\phi(\overline{x})$, $\mathbb M \models \phi(\overline{a'})$ if and only if $\mathbb M \models \phi(\overline{a})$.
	Since $\randgT$ admits quantifier elimination, we have $\mytp(\overline{a'}/A)=\mytp(\overline{a}/A)$.
	
	Condition (f) called independence over models says that, for every model $ M$, for every tuples $\overline{a}, \overline{b}, \overline{a'}, \overline{b'}$, if ${\overline{a}}\myindzero_{M}{\overline{b}}$, ${\overline{a}}\myindzero_{M}{\overline{a'}}$, ${\overline{b}}\myindzero_{M}{\overline{b'}}$ and $\mytp(\overline{a'}/M) =\mytp(\overline{b'}/M)$, then there exists a tuple $\overline{c}$ such that $\mytp(\overline{c}/M\overline{a}) =\mytp(\overline{a'}/M\overline{a})$, $\mytp(\overline{c}/M\overline{b}) =\mytp(\overline{b'}/M\overline{b})$ and ${\overline{c}}\myindzero_{M}{\overline{a}\overline{b}}$.
	
	Let $\overline{a'}=(a'_1,\ldots, a'_l)$ and $\overline{b'}=(b'_1,\ldots, b'_l)$.
	Observe that, for every $1 \leq i \leq l$ and $1 \leq j \leq n$, we have $\mathbb M \models C_j(a'_i) \Leftrightarrow \mathbb M \models C_j(b'_i)$.
	We also have $a'_i \in M$ if and only if $b'_i \in M$, and if $a'_i \in M$, we have $a'_i=b'_i$.
	We also have $\mathbb M \models R(m,a'_i) \leftrightarrow R(m,b'_i)$ for $1 \leq i \leq l$ and $m \in M$.
	In the same manner as the proof for condition (e), we can construct $\overline{c}=(c_1,\ldots, c_n)$ so that $c_i=a'_i=b'_i$ if $a'_i \in M$, $\mathbb M\models C_k(c_i) \leftrightarrow C_k(a'_i)$, $\mathbb M \models R(m,c_i) \leftrightarrow R(m,b'_i)$ for $m \in M$,  $\mathbb M \models R(\alpha,c_i) \leftrightarrow R(\alpha,a'_i)$ for $\alpha \in \overline{a}$, $\mathbb M \models R(\beta,c_i) \leftrightarrow R(\beta,b'_i)$ for $\beta \in \overline{b}$, and $c_i \notin \overline{a}\overline{b} \setminus M$ unless $a'_i \in M$.
	By quantifier elimination, we have $\mytp(\overline{c}/M\overline{a}) =\mytp(\overline{a'}/M\overline{a})$, $\mytp(\overline{c}/M\overline{b}) =\mytp(\overline{b'}/M\overline{b})$ and ${\overline{c}}\myindzero_{M}{\overline{a}\overline{b}}$.
\end{proof}

\section{Generic $\compgraph$-free $n$-partite graphs}\label{sec:free}

Let $n>1$ be a positive integer, and let $\overline{m}=(m_1,\ldots,m_n)$ be a sequence of positive integers.
We fix $n$ and $\overline{m}$ in this section.
Let $\compgraph$ be the graph consisting of $m_i$ vertices of color $i$ for $1 \leq i \leq n$ and every two vertices with different colors are adjacent.
For $I \subseteq \{1,\ldots,n\}$, let $\compgraph\langle I \rangle$ be the complete graph consisting of $m_c$ vertices of color $c$ for $c \in I$.
In other words, $\compgraph\langle I \rangle$ is a subgraph of $\compgraph$ constructed by removing all the vertices in $\compgraph$ of color $\notin I$.
For $1 \leq c \leq n$, $\subcompgraph^{c}$ is the graph constructed from $\compgraph$ by removing all the vertices of color $c$, that is, $\subcompgraph^{c} := \compgraph\langle \{1,\ldots,n\} \setminus \{c\} \rangle$.
Let $\frcolgT$ be the theory of \textit{$\compgraph$-free} $n$-partite graphs, that is, every model of $\frcolgT$ does not have a copy of $\compgraph$ as a subgraph.

Let $\frcolgK$ be the collection of finite models of $\frcolgT$ under the identification of structures isomorphic to each other.
\begin{lemma}\label{lem:bipartite}
	$\frcolgK$ does not possess the amalgamation property if $m_1 \geq 2$.
\end{lemma}
\begin{proof}
	We consider the case in which $n=2$.
	We can prove it for $n>2$ in the same manner.
	
	Let $C$ be the disjoint union of a vertex $p$ of color $2$ and a copy of the complete graph $K_{m_1-2,m_2}$.
	We suppose $p$ is not adjacent to any vertices in $C$.
	Let $A_i$ be the bipartite graph by adding a new vertex $q_i$ of color $1$ to $C$ for $i=1,2$.
	We join $q_1$ with every vertices of color $2$ in $C$ by edges, and join $q_2$ with every vertices of color $2$ except $p$ by edges.
	Assume for contradiction that there exists an amalgamation $D$ of $A_1$ and $A_2$ over $C$ in $\frcolgK$.
	Let $f_i:A_i \to D$ be the embedding over $C$.
	$f_1(q_1)$ is adjacent to $p$, but $f_2(q_2)$ is not.
	Therefore, we have $f_1(q_1) \neq f_2(q_2)$.
	$\{f_1(q_1), f_2(q_2)\} \cup (C \setminus \{p\})$ is isomorphic to $K_{m_1,m_2}$, which is absurd.
\end{proof}

Conant and Kruckman examine the model companion $\randfrcolgT$ of $\frcolgT$ in \cite{CK} when $n=2$.
We want to generalize some of their results to the case of $n>2$. 

Let $\frcolgTc$ be the extension of $\frcolgT$ such that a model $G$ of $\frcolgT$ is a model of $\frcolgTc$ if and only if, for every $1 \leq c \leq n$ and every tuples $\overline{v}_j$ of vertices of color $j$ of length $m_j$ for $1 \leq j \leq n$ with $j \neq c$, the following condition is satisfied: 
\begin{enumerate}
	\item[(*)] If $(\overline{v}_j\;|\; 1 \leq j \leq n, j \neq c)$ is a copy of $\subcompgraph^c$, there exist exactly $m_c-1$ vertices $y_1,\ldots, y_{m_c-1}$ such that every $y_k$ is fully adjacent to $(\overline{v}_j\;|\; 1 \leq j \leq n, j \neq c)$ for $1 \leq k < m_c$. 
\end{enumerate}

\begin{lemma}\label{lem:existence_free_completion}
	Let $G$ be a model of $\frcolgT$.
	There exists a model $F(G)$ of $\frcolgTc$ which is an extension of $G$ as an $n$-partite graph. 
\end{lemma}
\begin{proof}
	Put $G_0:=G$.
	We construct an extension $G_k$ of $ G$ by induction on $k$.
	Let $1 \leq l \leq n$.
	As an intermediate step, for every $1 \leq l \leq n$, we construct an extension $G_{k,l}$ of $ G_k$ by only adding vertices of color $l$.
	For convenience, we put $ G_{k,0}:= G_k$.
	
	Let $\mathfrak C_i$ be the collection of tuples $(\overline{v}_j\;|\; 1 \leq j \leq n, j \neq i)$ of tuples $\overline{v}_j$ of vertices in $G_{k,i-1}$ of color $j$ of length $m_j$ which makes a copy of $\subcompgraph^i$.
	Enumerate elements in $\mathfrak C_i$.
	We have $\mathfrak C_i=((\overline{v}^{\lambda}_j\;|\; 1 \leq j \leq n, j \neq i)\;|\; \lambda<|\mathfrak C_i|)$.
	We construct $G_{k,i,\lambda}$ so that $G_{k,i,\lambda} \subseteq G_{k,i,\lambda'}$ if $\lambda<\lambda'$ as follows:
	For $(\overline{v}^{\lambda}_j\;|\; 1 \leq j \leq n, j \neq i)$ in $\mathfrak C_i$, we add new vertices $\overline{v}_{k,i,\lambda}$ of color $i$ so that condition (*) is satisfied.
	We do not join new vertices with the vertices which do not appear in $(\overline{v}^{\lambda}_j\;|\; 1 \leq j \leq n, j \neq i)$.
	Put $G_{k,i,\lambda}=\bigcup_{\lambda'<\lambda} G_{k,i,\lambda'} \cup \overline{v}_{k,i,\lambda}$.
	It is obvious that $G_{k,i,\lambda}$ is $\compgraph$-free if $ G_{k,i,\lambda'}$ is so for $\lambda'<\lambda$.
	By transfinite induction on $\lambda<|\mathfrak C|$, $G_{k,i,\lambda}$ is $\compgraph$-free.

	We put $G_{k,i}:= \bigcup_{\lambda<|\mathfrak C|}G_{k,i,\lambda}$ and $G_k=\bigcup_{i=1}^n G_{k,i}$.
	Put $F( G):=\bigcup_{k=0}^{\infty}  G_k$.
	$F(G)$ is $\compgraph$-free.
	It is easy to prove that $F(G)$ satisfies condition (*).
	Therefore, $F(G)$ is a model of $\frcolgTc$.
\end{proof}

Let $B$ be a model of $\colorgT$.
Let $\mydiag_B(\overline{z})$ denote the set of all atomic and negated atomic formulas which are true in  $B$.
If $A$ is a substructure of $B$ and the variables $\overline{x}$ and $\overline{y}$ correspond to the vertices in $A$ and $B \setminus A$, respectively, we denote $\mydiag_{B}(\overline{z})$ by $\mydiag_{A,B}(\overline{x};\overline{y})$.
If $A$ is finite, we identify $\mydiag_{A}(\overline{z})$ with $\bigwedge_{\phi \in \mydiag_{A}(\overline{z})}\phi$.
We call a pair $(A,B)$ a \textit{safe pair} if, $A$ is a subgraph of a $\compgraph$-free graph $B$ and the following condition is satisfied:
\begin{itemize}
	\item For every $p \in B \setminus A$, $\{v \in A\;|\; B \models R(v,p)\}$ does not contain a copy of $\subcompgraph^c$, where $c=\mycol(p)$.
%	\item $|B \setminus A|=1$ and, for a unique vertex $p$ in $B \setminus A$, the graph $\{q \in A\;|\; B \models R(p,q)\}$ does not contain a copy of $\subcompgraph^c$, where $c$ is the color of $p$.
%	\item $|B \setminus A|>1$ and there exists an enumeration $v_1,\ldots,v_k$ of $B \setminus A$ such that $(Av_1\ldots v_{i-1},Av_1,\ldots v_i)$ is a safe pair for each $1 \leq i \leq k$.
\end{itemize}
%For a set of formulas $\Delta(\overline{x},\overline{y})$, the subset of $\Delta$ containing those formulas which only mention variables in $\overline{x}$ is denoted by $\widehat{\Delta}(\overline{x})$. 

Let $\randfrcolgT$ be the extension of $\frcolgTc$ by adding all the sentences of the form:
$$\forall \overline{x}\ \left(\mydiag_{A}(\overline{x}) \rightarrow \exists \overline{y}\ \mydiag_{A,B}(\overline{x};\overline{y})\right)$$ for every safe pair $(A,B)$ of finite graphs.
%From now on, $\mathbb M$ is a monster model of $\randfrcolgT$.

\begin{lemma}\label{lem:existence_free_completion2}
	Let $H$ be a model of $\frcolgT$.
	There exists a model $G(H)$ of $\randfrcolgT$ which is an extension of $H$ as an $n$-partite graph. 
\end{lemma}
\begin{proof}
	Consider the family $\mathfrak C$ of all the safe pairs $(A,B)$ of finite graphs under the identification of isomorphic graphs.
	Observe that $\mathfrak C$ is countable.
	Enumerate all the elements in $\mathfrak C$, say $(A_1, B_1), \ldots$.  
	
	Put $H_0:=H$.
	We define a model $H_k$ of $\frcolgT$ for $k>0$ by induction on $k$.
	As an intermediate step, for every $l \geq 1$, we construct a model $H_{k-1,l}$ of $\frcolgT$ such that, for each $\overline{a} \in  H_{k-1}$, there exists $\overline{b} \in  H_{k-1,l}$ such that $H_{k-1,l} \models (\mydiag_{ A_l}(\overline{a}) \rightarrow \mydiag_{ A_l, B_l}(\overline{a};\overline{b}))$.
	Enumerate tuples of elements $(\overline{a}_{\lambda}\;|\;\lambda < | H_{k-1}|)$ in $ H_{k-1}$ such that $ H_{k-1} \models \mydiag_{ A_l}(\overline{a}_{\lambda})$.
	We will construct $ N_{k-1,l,\lambda}$ so that 
	\begin{itemize}
		\item $N_{k-1,l,\lambda} \models \exists \overline{y}\ \mydiag_{A_l,B_l}(\overline{a}_{\lambda};\overline{y})$,
		\item $N_{k-1,l,\lambda}$ is a substructure of $N_{k-1,l,\lambda'}$ whenever $\lambda<\lambda'$, and 
		\item $N_{k-1,l,\lambda}$ is $\compgraph$-free.
	\end{itemize}
	Put $M_{k-1,l,\lambda}:=\bigcup_{\lambda'<\lambda} N_{k-1,l,\lambda'}$.
	Then, $M_{k-1,l,\lambda}$ is $\compgraph$-free.
	Let $N_{k-1,l,\lambda}:= M_{k-1,l,\lambda} \otimes_{\overline{a}_{\lambda}} B_l$ under the identification of $\overline{a}_{\lambda}$ with $ A_l$.
	We show that $ N_{k-1,l,\lambda}$ is $\compgraph$-free.
	Assume for contradiction that $ N_{k-1,l,\lambda}$ contains a copy $K$ of $\compgraph$.
	Put $ Q_1:= K \cap  M_{k-1,l,\lambda} \setminus \overline{a}_{\lambda}$, $ Q_2:= K \cap  B_l \setminus \overline{a}_{\lambda}$ and $ Q_3:= K \cap\overline{a}_{\lambda}$.
	By the definition of free amalgam, every vertex in $ Q_1$ is fully separated from $ Q_2$.
	Therefore, the vertices of $ Q_1$ and $ Q_2$ have an identical color, say $c$.
	We have $ Q_2 \neq \emptyset$ because $ M_{k-1,l,\lambda}$ is $\compgraph$-free.
	$ Q_3$ must contain a copy of $\subcompgraph^c$, and every vertex in $ Q_2$ is fully adjacent to this copy.
	This contradicts that $( A_l, B_l)$ is a safe pair.
	We have shown that the free amalgam $ N_{k-1,l,\lambda}$ is $\compgraph$-free.
	
	Put $H_{k,l}:=\bigcup_{\lambda}N_{k-1,l,\lambda}$.
	It is obvious that, for each $\overline{a} \in G_{k-1}$, there exists ${b} \in H_{k-1,l}$ such that $G_{k-1,l} \models (\mydiag_{ A_l}(\overline{a}) \rightarrow \mydiag_{A_l, B_l}(\overline{a};{b}))$.
	We put $ H_k:=F\left(\bigcup_{l=1}^{\infty} H_{k-1,l}\right)$, where $F(\cdot)$ is the operator defined in Lemma \ref{lem:existence_free_completion}.
	Put $G( H):=\bigcup_{k=0}^\infty  H_k$.
	It is a routine to prove that $G( H)$ is a model of $\randfrcolgT$.
	We omit the details.
\end{proof}

\begin{lemma}\label{lem:fr_model_complete}
	$\randfrcolgT$ is model complete.
\end{lemma}
\begin{proof}
	We use Robinson's test \cite[Lemma 3.2.7]{TZ}.
	Let $M_i$ be models of $\randfrcolgT$ for $i=1,2$ such that $M_1$ is a substructure of $M_2$.
	Let $\phi(\overline{x})$ be a quantifier-free $\colorgL(M_1)$-formula with $M_2 \models \exists \overline{x} \ \phi(\overline{x})$.
	We have to show that $M_1 \models \exists \overline{x} \ \phi(\overline{x})$.
	There exists a finite substructure $A$ of $M_1$ such that parameters appearing in $\phi(\overline{x})$ is contained in  $A$.
	Let $\overline{b}$ be the realization of $\phi(\overline{x})$ in $M_2$.
	Let $B:=A\overline{b}$ be the substructure of $M_2$.
	If we can find $\overline{b'} \in M_1$ such that $\overline{a}\overline{b'}$ is isomorphic to $B$, then we have $M_1 \models \phi(\overline{b'})$.
	We want to find such a $\overline{b'} \in M_1$.
	
	Let $k$ be the length of the tuple $\overline{b}$.
	By induction on $k$, we can easily reduce to the case in which $k=1$.
	If $(A,B)$ is a safe pair, we can find $\overline{b'} \in M_1$ such that $M_1 \models \phi(\overline{b'})$ by the definition of $\randfrcolgT$.
	Suppose not.
	We want to show that $\overline{b} \in M_1$.
	Let $c$ be the color of $\overline{b}$.
	Let $A':=\{q \in A\;|\; B \models R(\overline{b},q)\}$.
	The graph $A'$ contains a copy $K$ of $\subcompgraph^c$.
	Since $M_1$ is a model of $\frcolgTc$, there exist vertices $p_1,\ldots, p_{m_c-1}$ of color $c$ such that these vertices are fully adjacent to $K$.
	If $\overline{b} \notin M_1$, $M_2$ contains a copy of $\compgraph$, which is absurd.
	We have shown that $\overline{b} \in M_1$.
\end{proof}

In a special case, $\randfrcolgT$ admits quantifier elimination.
\begin{lemma}\label{lem:rand_qe}
	If $m_i=1$ for every $1 \leq i \leq n$, $\randfrcolgT$ admits quantifier elimination.
\end{lemma}
\begin{proof}
	We apply a standard QE test \cite[Theomre 3.2.5]{TZ}.
	Let $M$ and $N$ be two models of $\randfrcolgT$ and $A$ be a common substructure of $M$ and $N$.
	A basic formula is an atomic $\colorgL(A)$-formula or the negations of an atomic $\colorgL(A)$-formula.
	Let $\phi(x)$ be a conjunction of basic formulas $\theta_1(x),\ldots,\theta_k(x)$ with a single variable $x$ such that $M \models \exists x \ \phi(x)$.
	We show that $N \models \exists x \ \phi(x)$.
	
	Let $v \in M$ with $M \models \phi(v)$ and $\overline{a}$ be elements in $A$ appearing in the formula $\phi(x)$.
	The pair $(\overline{a},\overline{a}v)$ is a safe pair.
	By the definition of $\randfrcolgT$, there exists an embedding $f:\overline{a}v \to N$ fixing $\overline{a}$ pointwise.
	It is obvious that $N \models \phi(\overline{a},f(v))$.
\end{proof}

\begin{theorem}\label{thm:fr_model_companion}
	$\randfrcolgT$ is the model companion of both $\frcolgT$ and $\frcolgTc$.
\end{theorem}
\begin{proof}
	This theorem follows from Lemma \ref{lem:existence_free_completion}, Lemma \ref{lem:existence_free_completion2} and Lemma \ref{lem:fr_model_complete}.
\end{proof}

\begin{corollary}\label{cor:complete}
	$\randfrcolgT$ is complete.
\end{corollary}
\begin{proof}
	Let $M_i$ be models of $\randfrcolgT$ for $i=1,2$.
	The disjoint union $N$ of $M_1$ and $M_2$ is a model of $\frcolgT$.
	By Theorem \ref{thm:fr_model_companion}, there exists an extension $M$ of $N$, which is a model of $\randfrcolgT$.
	Since $M_i$ is a substructure of $M$ and $\randfrcolgT$ is model complete by Theorem \ref{thm:fr_model_companion}, $M_i \equiv M$ for $i=1,2$.
	This implies $\myth(M_1)=\myth(M)=\myth(M_2)$.
	We have shown that $\randfrcolgT$ is complete.
\end{proof}

\begin{lemma}\label{lem:amalgam_free_colored}
	Let $A$ be a model of $\frcolgTc$, and let $B_1$ and $B_2$ be models of $\frcolgT$ with $A \subseteq B_i$ for $i=1,2$.
	The free amalgam $M:=B_1 \otimes_{A} B_2$ is $\compgraph$-free.
	
	In addition, $\frcolgTc$ possesses the amalgamation property.
\end{lemma}
\begin{proof}
	We show that the free amalgam $M:=B_1 \otimes_{A} B_2$ is $\compgraph$-free.
	Assume for contradiction that $M$ contains a copy $K$ of $\compgraph$.
	Put $Q_i:=K \cap B_i \setminus A$ for $i=1,2$ and  $Q_3:=K \cap A$.
	By the definition of free amalgam, $Q_1$ is fully separated from $Q_2$.
	Therefore, the vertices of $Q_1$ and $Q_2$ have an identical color, say $c$.
	In particular, $Q_3$ contains a copy $K'$ of $\subcompgraph^c$. 
	Since $B_2$ is $\compgraph$-free, we have $Q_1 \neq \emptyset$. Let $q$ be a vertex in $Q_1$.
	Since $A$ is a model of $\frcolgTc$, there exists $m_c-1$ vertices $p_1,\ldots, p_{m_c-1}$ fully adjacent to $K'$.
	The subgraph $K' \cup \{p_1,\ldots,p_{m_c-1},q\}$ of $B_1$ is a copy of $\compgraph$, which is absurd.
	We have shown that $ M:=B_1 \otimes_{A} B_2$ is $\compgraph$-free.
	
	Let $ C:=F( B_1 \otimes_{ A}  B_2)$ be the graph constructed in Lemma \ref{lem:existence_free_completion}.
	$ C$ is a model of $\frcolgTc$.
	The embeddings $ B_i \to  C$ fixes $ A$ pointwise.
	This means that $\frcolgTc$ has the amalgamation property.
\end{proof}

\begin{corollary}\label{cor:fr_model_completion}
	$\randfrcolgT$ is the model completion of $\frcolgTc$.
\end{corollary}
\begin{proof}
	This corollary follows from Lemma \ref{lem:amalgam_free_colored}, Theorem \ref{thm:fr_model_companion} and \cite[Propositon 3.5.18]{ChanK}. 
\end{proof}

In the rest of this section, let $\mathbb M$ be a monster model of $\randfrcolgT$.

\begin{proposition}\label{prop:fr_monster_basic}
	The following assertions hold: 
	\begin{enumerate}
		\item[(1)] Let $ A$ and $ B$ be models of $\frcolgTc$ and $\frcolgT$, respectively.
		Let $f: A \to \mathbb M$ and $g: A \to  B$ be $\colorgL$-embeddings.
		Then, there exists an $\colorgL$-embedding $h: B \to \mathbb M$ such that $f=h \circ g$.  
		\item[(2)] Suppose $ A \subseteq \mathbb M$ is a model of $\frcolgTc$, and $C \subseteq  A$.
		Let $\overline{a} \in  A$ and $f: A \to \mathbb M$ is an $\colorgL$-embedding fixing $C$ pointwise.
		Let $ M$ be a small elementary submodel of $\mathbb M$ with $ A \subseteq  M$.
		Then, we can extend $f$ to an elementary embedding $g: M \to \mathbb M$, and $\mytp(\overline{a}/C)=\mytp(f(\overline{a})/C)$.
		\item[(3)] Let $A$ be a small subset of $\mathbb M$.
		The algebraic closure $\myacl(A)$ of $A$ coincides with the set of vertices of a smallest model of $\frcolgTc$ containing $A$.
	\end{enumerate}	
\end{proposition}
\begin{proof}
	(1) We may assume that $f$ is an inclusion.
	Let $M \subseteq \mathbb M$ be a small submodel of $\randfrcolgT$ containing $A$.
	By Lemma \ref{lem:existence_free_completion2}, there exists a model $N$ of $\randfrcolgT$ containing $B$.
	By Lemma \ref{lem:amalgam_free_colored}, $N':=M \otimes_A N$ is $\compgraph$-free.
	We can embed $N'$ into $\mathbb M$ by Lemma \ref{lem:existence_free_completion2} and saturation of $\mathbb M$.
	Therefore, we assume that $N'$ is a substructure of $\mathbb M$.
	Let $M'$ and $B'$ be copies of $M$ and $B$ in $N'$, respectively.
	Let $\sigma:M' \to M$ and $g':B \to B'$ be isomorphisms.
	We can extend $\sigma$ to an automorphism $\overline{\sigma}:\mathbb M \to \mathbb M$ by \cite[Corollary 6.1.8]{TZ} because $\mathbb M$ is a monster model.
	The map $\overline{\sigma} \circ g'$ is a desired embedding.
	
	(2) Let $ M$ be a small elementary submodel of $\mathbb M$ with $ A \subseteq  M$.
	By (1), we can extend $f$ to an embedding $g: M \to \mathbb M$.
	Since $\randfrcolgT$ is model complete by Lemma \ref{lem:fr_model_complete}, $g$ is an elementary embedding.
	Therefore, $\mytp(\overline{a}/C)=\mytp(g(\overline{a})/g(C))=\mytp(f(\overline{a})/C)$.
	
	(3) As an intermediate step, we show the following claim:
	\medskip
	
	\textbf{Claim A.} $\myacl(A)$ is a model of $\frcolgTc$.
	\begin{proof}[Proof of Claim]
		Let $\overline{v}$ be a copy of $\subcompgraph^c$ in $\myacl(A)$ for some color $c$.
		Since $\mathbb M$ is a model of $\frcolgTc$, there exist exactly $m_c-1$ vertices $\overline{w}$ in $\mathbb M$ of color $c$ fully adjacent to $\overline{v}$.
		We have $\overline{w} \in \myacl(\myacl(A))$.
		By \cite[Lemma 5.6.6]{TZ}, the equality $\myacl(A)=\myacl(\myacl(A))$ holds.
		This means $\overline{w} \in \myacl(A)$, which shows that $\myacl(A)$ is a model of $\frcolgTc$.
	\end{proof}
	It is easy to show that the intersection of all the models of $\frcolgTc$ containing $A$ is also a model of $\frcolgTc$.
	We omit the proof.
	Let $ C$ be this intersection.
	Claim A implies $ C \subseteq \myacl(A)$.
	Therefore, we may suppose that $A= C$ without loss of generality.
	In particular, we may assume that $A$ is a model of $\frcolgTc$.
	
	Assume for contradiction that $ C \neq \myacl(A)$.
	Choose $b_1 \in \myacl(A) \setminus  C$.
	The type $\mytp(b_1/A)$ has finitely many realization, say $b_1,\ldots, b_k$.
	Consider the free amalgam $\myacl(A) \otimes_{ C}\myacl(A)$ which is a model of $\frcolgT$ by Lemma \ref{lem:amalgam_free_colored}.
	There exists an $\colorgL$-embedding $\myacl(A) \otimes_{ C}\myacl(A) \hookrightarrow \mathbb M$ fixing $C$ pointwise by (1).
	This induces two $\colorgL$-embeddings $g,h:\myacl(A) \to \mathbb M$ fixing $ C$ pointwise such that $g(\myacl(A) \setminus  C) \cap h(\myacl(A) \setminus  C)=\emptyset$.
	By (2), $g(b_1)\ldots, g(b_k), h(b_1),\ldots, h(b_k)$ are solutions of $\mytp(b_1/A)$, which contradicts that $\mytp(b_1/A)$ has only $k$ solutions.
\end{proof}

\begin{definition}
	A formula $\phi(\overline{x})$ is called a \textit{basic existential formula} if there exists $ A,  B \in \frcolgK$ such that $ A \subseteq B$ and $\phi(\overline{x})$ is equivalent to $\exists \overline{y}\ \mydiag_{ A,  B}(\overline{x'},\overline{y})$ modulo $\randfrcolgT$, where $\overline{x'}$ is a subtuple of $\overline{x}$.
\end{definition}

The following proposition coincides with \cite[Proposition 2.17]{CK} if $n=2$.
We give a more elementary proof than that of \cite[Proposition 2.17]{CK}.

\begin{proposition}\label{prop:almost_qe}
	Every formula is equivalent to a disjunction of finitely many basic existential formulas modulo $\randfrcolgT$.
\end{proposition}
\begin{proof}
	Let $\phi(\overline{x})$ be an arbitrary formula.
	Since $\randfrcolgT$ is model complete by Lemma \ref{lem:fr_model_complete}, $\neg\phi(\overline{x})$ is equivalent to a universal formula by \cite[Lemma 3.2.7]{TZ}.
	Therefore, $\phi$ is equivalent to an existential formula.
	We may assume that $\phi(\overline{x})=\exists \overline{y}\ \psi(\overline{x},\overline{y})$, where $\psi(\overline{x},\overline{y})$ is a quantifier-free formula.
	The formula $\psi(\overline{x},\overline{y})$ is of the form $\psi(\overline{x},\overline{y})=\bigvee_{i=1}^k \chi_i(\overline{x},\overline{y})$, where $\chi_i (\overline{x},\overline{y})$ is a conjunction of atomic formulas and negated atomic formulas.
	Let $\operatorname{var}(\psi)$ be the set of variables appearing in $\psi(\overline{x},\overline{y})$.
	Replace 
	\begin{align*}
		&\psi(\overline{x},\overline{y}) \wedge \bigwedge_{v \in \operatorname{var}(\psi)}\left(\bigvee_{i=1}^nC_i(v)\right) \\
		&\qquad \wedge \bigwedge_{v,w \in \operatorname{var}(\psi), v \neq w}((v=w) \vee ( v \neq w)) \wedge (R(v,w) \vee \neg R(v,w))),
	\end{align*}
	which is equivalent to $\psi(\overline{x},\overline{y})$ modulo $\randfrcolgT$, with $\psi(\overline{x},\overline{y})$.
	We may assume that, for every $v \in \operatorname{var}(\psi)$, $\chi_i (\overline{x},\overline{y})$ implies $C_{k_v}(v)$ for some $1 \leq k_v \leq n$ thanks to this replacement.
	In the same manner, we may assume that, for every $v,w \in \operatorname{var}(\psi)$ with $v \neq w$, $\chi_i (\overline{x},\overline{y})$ implies either the formula $v=w$ or the formula $v \neq w$.
	We can employ the similar assumption for $R(v,w)$ and $\neg R(v,w)$.
	Substitute $w$ into every occurrence $v$ in $\chi_i$ if $\chi_i$ implies the formula $v=w$, then $\chi_i$ is of the form $\mydiag_{A_i,B_i}(\overline{x'}_i,\overline{y'}_i)$ for some graphs $A_i$ and $B_i$ with $A_i \subseteq B_i$, where $\overline{x'}_i$ and $\overline{y'}_i$ are subsequences of $\overline{x}$ and $\overline{y}$, respectively.
	Since $\phi(\overline{x})$ is equivalent to $\exists \overline{y}\ \bigvee_{i=1}^k \chi_i(\overline{x},\overline{y})$, $\phi(\overline{x})$ is equivalent to $\bigvee_{i=1}^k (\exists \overline{y}\  \chi_i(\overline{x},\overline{y}))$, and this is equivalent to $\bigvee_{i=1}^k (\exists \overline{y'}_i\  \mydiag_{A_i,B_i}(\overline{x'}_i,\overline{y'}_i))$.
\end{proof}

\begin{definition}[{\cite{Chernikov,Shelah}}]
	A theory $T$ is $\mytptwo$ if there exist a formula $\phi(\overline{x},\overline{y})$ and a collection of tuples in a model $(\overline{a}_{ij}\;|\;i,j<\omega)$ such that, for every $\sigma:\omega \to \omega$, $\{\phi(\overline{x},\overline{a}_{i\sigma(i)})\}$ is consistent, and $\phi(\overline{x},\overline{a}_{i,j}) \wedge  \phi(\overline{x},\overline{a}_{i,j'})$ is inconsistent for $i<\omega$ and $j<j'<\omega$.
	We say that $T$ is $\myntptwo$ if $T$ is not $\mytptwo$.
	
%	Two elements $\eta,\mu \in 2^{<\omega}$ are comparable if $\eta$ is a concatenation of $\mu$ with some sequence or $\mu$ is a concatenation of $\eta$ with some sequence.
%	For $\mu \in 2^{\omega}$, $\mu|_k$ is the sequence consisting of the first $k$ elements in $\mu$.
%	$T$ is $\mysop_2$ if there exist a formula $\phi(\overline{x},\overline{y})$ and $(\overline{a}_{\eta}\;|\;\eta \in 2^{<\omega})$ be a collection of tuples in a model such that, for every $\eta,\mu \in 2^{<\omega}$, $\{\phi(\overline{x},\overline{a}_{\eta|_k})\;|\;k \leq |\eta|\}$ is consistent, and $\phi(\overline{x},\overline{a}_{\eta}) \wedge  \phi(\overline{x},\overline{a}_{\mu})$ is inconsistent whenever $\eta$ and $\mu$ are incomparable.
%	We say that $T$ is $\mynsop_2$ if $T$ is not $\mysop_2$.
	
	For $k \geq 3$, $T$ is $\mysop_k$ if there exist a formula $\phi(\overline{x},\overline{y})$ with $|\overline{x}|=|\overline{y}|$ and a collection $(\overline{a}_i\;|\;i<\omega)$ of tuples in a model such that
	\begin{itemize}
		\item $\mathbb M \models \phi(\overline{a}_i,\overline{a}_j)$ if $i<j$;
		\item $\phi(\overline{x}_1,\overline{x}_2) \wedge \cdots \wedge \phi(\overline{x}_{k-1},\overline{x}_k) \wedge \phi(\overline{x}_k,\overline{x}_1)$ is inconsistent.
	\end{itemize}
	We say that $T$ is $\mynsop_k$ if $T$ is not $\mysop_k$.
%	
%	$T$ is $\mysop$ if there exist a formula $\phi(\overline{x},\overline{y})$ and $(\overline{a}_i\;|\;i<\omega)$ be a collection of tuples in a model such that $\phi(\overline{x},\overline{a}_i)\wedge \phi(\overline{x},\overline{a}_j)$ is consistent if and only if $i<j$.
%	We say that $T$ is $\mynsop$ if $T$ is not $\mysop$.
\end{definition}

\begin{proposition}\label{prop:fr_TP2}
	If $n>2$, $\randfrcolgT$ is $\mytptwo$.
\end{proposition}
\begin{proof}
	Let $\overline{c}$ be the complete graph of consisting $m_i$ vertices of color $i$ for $3 \leq i \leq n$.
	Let $\overline{b}$ be the tuple of vertices of color $2$ of length $m_2$.
	Let $(\overline{b}_{i,j}\overline{c}_{i,j}\;|\;i,j<\omega)$ be copies of $\overline{b}\overline{c}$.
	We join vertices by edges so that $\overline{b}_{i,j}$ is fully adjacent to $\overline{c}_{i',j'}$ if and only if $i=i'$ and $j \neq j'$, and they are fully separated from each other otherwise.
	Since the graph constructed above does not have vertices of color $1$, it is a model of $\frcolgT$.
	Therefore, we can embed it into $\mathbb M$ by Theorem \ref{thm:fr_model_companion}.
	
	Let $\overline{x}$ be the tuple of variables of length $m_1$.
	Let $\phi(\overline{x},\overline{y},\overline{z})$ denote the formula saying the following:
	\begin{itemize}
		\item $\overline{x}$ is a tuple of vertices of color $1$ of length $m_1$;
		\item $\overline{y}$ is a tuple of vertices of color $2$ of length $m_2$;
		\item $\overline{z}$ is a copy of $\overline{c}$;
		\item $\overline{x}$ is fully adjacent to $\overline{y}$ and $\overline{z}$.
	\end{itemize}
	
	Let $\sigma: \omega \to \omega$ be an arbitrary map.
	We show that $\{\phi(\overline{x},\overline{b}_{i,\sigma(i)},\overline{c}_{i,\sigma(i)})\;|\;i<\omega\}$ is consistent.
	By saturation of $\mathbb M$, we have only to show that, for every $k<\omega$, $\{\phi(\overline{x},\overline{b}_{i,\sigma(i)},\overline{c}_{i,\sigma(i)})\;|\;i<k\}$ has a solution.
	%Let $\overline{x}=(x_1,\ldots,x_{m_1})$.
	Put $K_{\sigma,k}:= \overline{b}_{1,\sigma(1)}\overline{c}_{1,\sigma(1)}\ldots\overline{b}_{k,\sigma(k)}\overline{c}_{k,\sigma(k)}$.
	$(K_{\sigma,k}, K_{\sigma,k}\overline{x})$ is a safe pair. % for each $1 \leq i \leq k$.
	Therefore, we can find a solution using the axiom of $\randfrcolgT$.
	
	For every $i<\omega$ and $j, k<\omega$ with $j \neq k$, $\phi(\overline{x},\overline{b}_{i,j},\overline{c}_{i,j}) \wedge \phi(\overline{x},\overline{b}_{i,k},\overline{c}_{i,k})$ is inconsistent.
	In fact, if this formula has a solution $\overline{a}$, the graph $\overline{a}\overline{b}_{i,j}\overline{c}_{i,k}$ is a copy of $\compgraph$, which is a contradiction. 
\end{proof}
%
%\begin{remark}
%	Let $K$ be a copy of $K_{\overline{m}}\langle \{3,4,\ldots,n\} \rangle$ in $\mathbb M$.
%	It is easy to see that $M:=\{v \in \mathbb M\;|\; v \text{ is fully adjacent to }K\}$ is a model of $\randfrcolgT$ for $n=2$.
%	We can translate some results in \cite{CK} directly into the cases of $n>2$ using the above fact. 
%	For instance, we can show Proposition \ref{prop:fr_TP2} using this fact and \cite[Proposition 4.18]{CK} under the additional assumptions $m_1>1$ and $m_2>1$.
%	We can also show that there exists a formula which forks over $M$, but does not divide over $M$ in the same manner using \cite[Proposition 4.25]{CK}.
%\end{remark}

\begin{proposition}\label{prop:fr_SOP3}
	If $n>2$, $\randfrcolgT$ is $\mysop_3$.
\end{proposition}
\begin{proof}
	Let $\overline{a}$ be the tuple of vertices of color $1$ of length $m_1$.
	Let $\overline{b}$ be the tuple of vertices of color $2$ of length $m_2$.
	Let $\overline{c}$ be the complete graph of consisting $m_i$ vertices of color $i$ for $3 \leq i \leq n$.
	Let $(\overline{v}_j:=\overline{a}_j\overline{b}_{j}\overline{c}_{j}\;|\;j <\omega)$ be copies of $\overline{a}\overline{b}\overline{c}$.
	We join vertices by edges so that, if $i<j$,  $\overline{a}_i$ is fully adjacent to $\overline{b}_j$, $\overline{b}_i$ is fully adjacent to $\overline{c}_j$, and $\overline{c}_i$ is fully adjacent to $\overline{a}_j$,
	We show that $(\overline{v}_i\;|\;i<\omega)$ is $\compgraph$-free.
	Assume for contradiction that $(\overline{v}_i\;|\;i<\omega)$ contains a copy of $\compgraph$.
	Let $j_{\max}:=\max\{j \;|\; K \cap \overline{v}_j \neq \emptyset\}$.
	Choose $u \in \overline{v}_{j_{\max}} \cap K$.
	We only consider the case where $u$ is of color $1$, but we can lead to a contradiction in a similar manner in the other cases.
	There is no vertex of color $2$ in $\overline{v}_j$ adjacent to $u$ for $j \leq j_{\max}$, which contradicts that $K \subseteq \bigcup_{j \leq j_{\max}}\overline{v}_j$.% and there exists a vertex in $K$ of color $2$ adjacent to $u$.
	We have proved that $(\overline{v}_i\;|\;i<\omega)$ is $\compgraph$-free.
	By Theorem \ref{thm:fr_model_companion}, we may assume that $(\overline{v}_i\;|\;i<\omega)$ is a substructure of $\mathbb M$.
	
	Let $\overline{w}_i=\overline{x}_i\overline{y}_i\overline{z}_i$ for $i=1,2$.
	Let $\phi(\overline{w}_1,\overline{w}_2)$ denote the formula saying the following:
	\begin{itemize}
		\item $\overline{x}_1$ and $\overline{x}_2$ are tuples of vertices of color $1$ of length $m_1$;
		\item $\overline{y}_1$ and $\overline{y}_2$ are tuples of vertices of color $2$ of length $m_2$;
		\item $\overline{z}_1$ and $\overline{z}_2$ are copies of $\overline{c}$;
		\item $\overline{x}_1$ is fully adjacent to $\overline{y}_2$, $\overline{y}_1$ is fully adjacent to $\overline{z}_2$, and $\overline{z}_1$ is fully adjacent to $\overline{x}_2$.
	\end{itemize}
	We have $\mathbb M \models \phi(\overline{v}_i,\overline{v}_j)$ if $i < j$.
	We show that $\phi(\overline{w}_1,\overline{w}_2) \wedge \phi(\overline{w}_2,\overline{w}_3) \wedge \phi(\overline{w}_3,\overline{w}_1)$ is inconsistent.
	Assume for contradiction that it has a realization, say $\overline{v'}_1,\overline{v'}_2,\overline{v'}_3$.
	Let $\overline{a'}_i$ be the tuple of vertices of color $1$ in $\overline{v'}_i$, $\overline{b'}_i$ be the tuple of vertices of color $2$ in $\overline{v'}_i$, $\overline{c'}_i$ be the tuple of vertices of color $\neq 1,2$ in $\overline{v'}_i$ for $1 \leq i \leq 3$.
	It is obvious that $\overline{a'}_1\overline{b'}_2\overline{c'}_3$ is a copy of $\compgraph$, which is absurd. 
\end{proof}

\begin{proposition}\label{prop:fr_NSOP4}
	If $n>2$, $\randfrcolgT$ is $\mynsop_4$.
\end{proposition}
\begin{proof}
	Assume for contradiction that $\randfrcolgT$ is $\mysop_4$.
	Let $\phi(\overline{x},\overline{y})$ and $(\overline{a}_i\;|\;i<\omega)$ be witnesses of $\mysop_4$, that is $\psi(\overline{x}_1,\overline{x}_2,\overline{x}_3,\overline{x}_4):=\phi(\overline{x}_1,\overline{x}_2) \wedge \phi(\overline{x}_2,\overline{x}_3) \wedge \phi(\overline{x}_3,\overline{x}_4) \wedge \phi(\overline{x}_4,\overline{x}_1)$ is inconsistent and $\mathbb M \models \phi(\overline{a}_i,\overline{a}_j)$ if $i<j$.
	By \cite[Lemma 5.1.3]{TZ}, we may assume that $(\overline{a}_i\;|\;i<\omega)$ is an indiscernible sequence over $\emptyset$.
	By Proposition \ref{prop:almost_qe}, there exist basic existential formulas $\chi_1(\overline{x},\overline{y}), \ldots,  \chi_k(\overline{x},\overline{y})$ such that $\phi(\overline{x},\overline{y})=\bigvee_{i=1}^k \chi_i(\overline{x},\overline{y})$.
	We may assume that $\mathbb M \models \chi_1(\overline{a}_1,\overline{a}_2)$ without loss of generality.
	Since $(\overline{a}_i\;|\;i<\omega)$ is indiscernible, we have $\mathbb M \models  \chi_1(\overline{a}_i,\overline{a}_j)$ for $i<j$.
	It is obvious that $\chi_1(\overline{x}_1,\overline{x}_2) \wedge \chi_1(\overline{x}_2,\overline{x}_3)  \wedge \chi_1(\overline{x}_3,\overline{x}_4) \wedge \chi_1(\overline{x}_4,\overline{x}_1)$ is inconsistent.
	Therefore, we may assume that $\phi$ is a basic existential formula by replacing $\phi$ with $\chi_1$.
	
	Since $\phi$ is a basic existential formula, there exists a pair of graphs $(A,B)$ such that $\phi(\overline{x},\overline{y})=\exists \overline{z}\ \mydiag_{A,B}(\overline{x}{}^\frown\overline{y},\overline{z})$.
	We put $\zeta(\overline{x}, \overline{y},\overline{z}):=\mydiag_{A,B}(\overline{x}{}^\frown\overline{y},\overline{z})$.
	There is a model $C=(\overline{b}_1,\overline{b}_2,\overline{b}_3, \overline{b}_4,\overline{c}_{12}, \overline{c}_{23}, \overline{c}_{34},\overline{c}_{41})$ of $\colorgT$ satisfying the following conditions:
	\begin{enumerate}
		\item[(1)] $\mathbb M \models \zeta(\overline{b}_i,\overline{b}_j,\overline{c}_{ij})$ holds for $(i,j)=(1,2), (2,3), (3,4), (4,1)$;
		\item[(2)] Vertices in $C$ are minimally adjacent, that is, two vertices $p$ and $q$ in $C$ are adjacent to each other exactly when condition (1) does not hold if $p$ and $q$ are separated.
	\end{enumerate}
	Assume for contradiction that $C$ is $\compgraph$-free.
	We can embed $C$ into $\mathbb M$ by Theorem \ref{thm:fr_model_companion}.
	This means that $\psi$ has a realization, which contradicts that $\psi$ is inconsistent.
	Therefore, $C$ contains a copy of $\compgraph$, say $K$.
	Put $K_i:=K \cap \overline{b}_i$ and $K_{ij}:=K \cap \overline{c}_ij$.
	
	Put $$\eta(\overline{x}_1,\overline{x}_2,\overline{x}_3,\overline{x}_4):=\phi(\overline{x}_1,\overline{x}_2) \wedge \phi(\overline{x}_2,\overline{x}_3) \wedge \phi(\overline{x}_3,\overline{x}_4).$$
	Observe that $\eta(\overline{x}_1,\overline{x}_2,\overline{x}_3,\overline{x}_4)$ is consistent because $\mathbb M \models \eta(\overline{a}_1,\overline{a}_2,\overline{a}_3,\overline{a}_4)$.
	
	We consider separate cases. 
	Put $\mathcal I:=\{(1,2), (2,3), (3,4), (4,1)\}$.
	First consider the case in which $K_i=K_{ij}=\emptyset$ for some $(i,j)\in \mathcal I$.
	We may assume that $(i,j)=(4,1)$ without loss of generality.
	The formula $\eta(\overline{x}_1,\overline{x}_2,\overline{x}_3,\overline{x}_4)$ is inconsistent in $\randfrcolgT$, which is absurd.  
	
	Next suppose at least one of $K_i$ and $K_{ij}$ is not empty for every $(i,j)\in \mathcal I$.	
	Let us consider the case in which $K_2 \neq \emptyset$.
	By the assumption, at least one of $K_4$ and $K_{41}$ is not empty.
	Let $L \in \{K_4,K_{34},K_{41}\}$ with $L \neq \emptyset$.
	Since $K_2$ is fully separated with $L$, $K_2$ and $L$ are monochromatic and have the same color, say $c$.
	Since $K_{12}$ and $K_{23}$ are fully separated with $L$, each of them is either empty or is monochromatic of color $c$.
	If $K_1 \neq \emptyset$ and $K_3 \neq \emptyset$, they are monochromatic and have the same color, say $c'$, because they are fully separated with each other.
	This implies that $K$ has only two colors, which contradicts $n>2$.
	We have $K_1 = \emptyset$ or $K_3 = \emptyset$.
	Assume $K_1 = \emptyset$ without loss of generality.
%	If $K_{41}=\emptyset$, $\eta(\overline{x}_1,\overline{x}_2,\overline{x}_3,\overline{x}_4)$ is inconsistent, which is absurd.
%	Therefore, $K_{41} \neq \emptyset$.
	Assume for contradiction that $K_{12} \neq \emptyset$.
	Since $K_{12}$ and $K_2$ are monochromatic of color $c$, they are fully separated.
	Recall that $K_1=\emptyset$ and $K_{12}$ is fully separated with $K_i$ and $K_{ij}$ other than $K_1$ and $K_2$.
	Therefore, vertices in $K_{12}$ cannot be adjacent to vertices in $K \setminus K_{12}$, which contradicts the definition of complete graphs. 
	We have $K_{12}=\emptyset$.
	However, we have $K_1=K_{12}=\emptyset$, which contradicts the case hypothesis.
	
	Let us consider the case in which $K_2 = \emptyset$ and $K_{23} \neq \emptyset$.
	If $K_1=\emptyset$, $K_{12}=\emptyset$ because $K_{12}$ is fully separated with $K_i$ and $K_{ij}$ other than $K_1$ and $K_2$.
	However, this contradicts the case hypothesis.
%	This implies that $\eta(\overline{x}_2,\overline{x}_3,\overline{x}_4,\overline{x}_1)$ is inconsistent, which is absurd.
	Next suppose $K_1 \neq \emptyset$.
	Since $K_1$ is fully separated with $K_{23}$, they are monochromatic and have the same color, say $c$.
	If $K_{12}=\emptyset$, $\eta(\overline{x}_2,\overline{x}_3,\overline{x}_4,\overline{x}_1)$ is inconsistent, which is absurd.
	If $K_{12} \neq \emptyset$, $K_{12}$ is monochromatic and of color $c$ because $K_{12}$ is fully separated with $K_{23}$.
	Since $K_1$ and $K_{12}$ are monochromatic and of color $c$, they are fully separated with each other.
	In addition, we have $K_2=\emptyset$ by the case hypothesis.
	Vertices in $K_{12}$ cannot be adjacent to vertices in $K \setminus K_{12}$, which contradicts the definition of complete graphs. 
\end{proof}

\begin{definition}[{\cite{Adler,TZ}}]
	Let $\mathcal L$ be a language and $T$ be a complete $\mathcal L$-theory.
	Let $\mathbb M$ be a monster model.
	Let $A \subseteq \mathbb M$ be a small set.
	Let $\phi(\overline{x},\overline{y})$ be an $\mathcal L(A)$-formula.
	We say that a formula $\phi(\overline{x},\overline{b})$ \textit{($k$-)divides} over $A$ if there exists a sequence $(\overline{b}_i\;|\;i<\omega)$ of realizations of $\mytp(\overline{b}/A)$ such that $(\phi(\overline{x},\overline{b}_i)\;|\;i<\omega)$ is $k$-inconsistent, that is, for any $k$-element subset $I$ of $\mathbb N$, $\bigwedge_{i \in I} \phi(\overline{x},\overline{b}_i)$ is inconsistent.
	A set of formula $\pi(\overline{x})$ \textit{divides} over $A$ if $\pi(\overline{x})$ implies some  $\phi(\overline{x},\overline{b})$ which divides over $A$.
	The set of formula $\pi(\overline{x})$ \textit{forks} over $A$ if $\pi(\overline{x})$ implies a disjunction $\bigvee_{i=1}^l \phi_i(\overline{x},\overline{b}_i)$ of formulas $\phi_i(\overline{x},\overline{b}_i)$ each dividing over $A$.

	Let $A,B,C$ be subsets of $\mathbb M$.
	We write ${A}\myindalg_{C}{B}$ if and only if $\myacl(AC) \cap \myacl(BC)=\myacl(C)$.
	We write ${A}\myinddiv_{C}{B}$ if and only if one of (both of) the following equivalent conditions is satisfied:
	\begin{itemize}
		\item For any finite tuples $\overline{a}$ from $A$, $\mytp(\overline{a}/BC)$ does not imply an $\mathcal L(C)$-formula $\phi(\overline{x},\overline{b})$ for $\overline{b} \subseteq B$ which divides over $C$;
		\item For any $C$-indiscernible sequence $(\overline{b}_i\;|\ i<\omega)$ with $\overline{b}_0 \in BC$, there exists $A'$ such that $A' \equiv_{BC} A$ and the sequence $(\overline{b}_i\;|\ i<\omega)$ is $A'C$-indiscernible.
	\end{itemize}
	Here, $A' \equiv_{BC} A$ means that there is an automorphism of the monster model fixing $BC$ pointwise and maps $A'$ to $A$.
	We write ${A}\myindfork_{C}{B}$ if and only if one of (both of) the following equivalent conditions is satisfied:
	\begin{itemize}
		\item For any finite tuples $\overline{a}$ from $A$, $\mytp(\overline{a}/BC)$ does not imply an $\mathcal L(C)$-formula $\phi(\overline{x},\overline{b})$ for $\overline{b} \subseteq B$ which forks over $C$;
		\item For any small subset $\widehat B$ of $\mathbb M$ containing $B$, there exists $A'$ such that $A' \equiv_{BC}A$ and $A' \myinddiv_C\widehat{B}$.
	\end{itemize}
\end{definition}

\begin{definition}\label{def:fine_cover}
	Let $\overline{k}=(k_1,\ldots,k_n)$ be a tuple of nonnegative integers.
	Let $B,C$ be graphs with $B \cap \myacl(C)=\emptyset$.
	We say that $B$ has a \textit{nice $\overline{k}$-cover over $C$} if there exist subgraphs $U_1,\ldots U_l$ of $B$ satisfying the following:
	\begin{enumerate}
		\item[(i)] $U_i$ is a complete graph (possibly a singleton) for $1 \leq i \leq l$, and the equality $B=\bigcup_{i=1}^l U_i$ holds.
		\item[(ii)] Put $\mathfrak n_{ij}:=|\{v \in U_i\;|\; \mycol(v)=j\}|$.
		Then, $\sum_{i=1}^l \mathfrak n_{ij}=k_j$ for $1 \leq j \leq n$.
		\item[(iii)] Let $M_i$ be copies of $B \cup \myacl(C)$ with $M_0=B \cup \myacl(C)$ for $i<\omega$.
		The copy of $U_j$ in $M_i$ is denoted by $U_{ji}$. 
		We construct graphs $N_i$ for $i<\omega$ so that $M_{i'}$ is a subgraph of $N_i$ for $i' \leq i$  as follows:
		Put $N_0:= M_0$.
		For $i>0$, $N_i$ be the graph constructed from $N_{i-1} \otimes_{\myacl(C)} M_i$ by joining vertices in $M_i$ with vertices in $N_{i-1}$ by edges so that $U_{ij}$ is fully adjacent to $U_{i'j'}$ for every $i'<i$ and $1 \leq j'<j \leq l$, but $U_{ij}$ is fully separated from $U_{i'j'}$ for every $i'<i$ and $1 \leq j< j' \leq l$. 
		Then, $N_i$ is $\compgraph$-free for all $i<\omega$.
	\end{enumerate}
\end{definition}

\begin{theorem}\label{thm:dividing}
	Suppose $n>2$.
	In $\randfrcolgT$, $A \myinddiv_CB$ if and only if $A \myindalg_C B$ and there are no graphs $G$ in $\myacl(AC) \cup \myacl(BC)$ called exceptional graphs (with respect to $C$) satisfying one of the following conditions, where $G[c]:=\{v \in G\;|\; \mycol(v)=c\}$ for $1 \leq c \leq n$, $G_A:=G \cap \myacl(AC) \setminus \myacl(C)$, $G_B:=G \cap \myacl(BC) \setminus \myacl(C)$ and $G_C:=G \cap \myacl(C)$.
	\begin{enumerate}
		\item[(1)] Exceptional graphs $G$ of type $1$ satisfy the following:
		\begin{enumerate}
			\item[(a)] $G_A \neq \emptyset$ and $G_B \neq \emptyset$.
			The inequality $1 \leq |G[c]| \leq m_c$ holds for every $1 \leq c \leq n$, and
			$|G[c]| =m_c$ holds for $c \notin \mycol(G_B)$.
			\item[(b)] Every two vertices $p$ and $q$ in $G$ are adjacent to each other unless $p,q \in G_B$.
			\item[(c)] Put $k_c=m_c-|\{v \in G_A \cup G_C\;|\;\mycol(v)=c\}|$ for $1 \leq c \leq n$ and $\overline{k}=(k_1,\ldots,k_n)$.
			$G_B$ has a nice $\overline{k}$-cover over $C$.
		\end{enumerate} 
		\item[(2)] Exceptional graphs $G$ of type $2$ satisfy the following:
		\begin{enumerate}
			\item[(a)] $G_A \neq \emptyset$ and $G_B \neq \emptyset$;
			\item[(b)] $|\mycol(G_A)|=1$, $\mycol(G_A) \cap \mycol(G_B)=\emptyset$, and $1 \leq |G[c]| \leq m_c$ for $1 \leq c \leq n$, $|G[c]| \neq m_c$ if $c \in \mycol(G_A)$ and $|G[c]|=m_c$ if $c \in \mycol(G_C) \setminus \mycol(G_B)$;
			\item[(c)] %$G_A$ is fully adjacent to $G_B \cup G_C$;
			Every two vertices $p$ and $q$ in $G$ are adjacent to each other unless $p,q \in G_B$.
%			\item[(d)] $G$ does not contain a copy of $\subcompgraph^{c_A}$, where $c_m$ is the unique color in $\mycol(G_A)$.
			\item[(d)] Let $\mycol(G_A)=\{c_A\}$. There are no $v \in \myacl(C)$ of color $c_A$ fully adjacent to $G_B \cup G_C$ other than the vertices in $G_C$.
			\item[(e)] Put $k_c=0$ if $c =c_A$ and $k_c:=m_c-|\{v \in G_C\;|\;\mycol(v)=c\}|$ otherwise.
			Put $\overline{k}=(k_1,\ldots,k_n)$.
			$G_B$ has a nice $\overline{k}$-cover over $C$.
%			\item[(f)] Put $R_G:=\{v \in \myacl(C)\;|\; \mycol(v)=c_A, v \text{ is fully adjacent to } G_A \cup G_C\}$. Then, $|R_G| <m_{c_A}$.
		\end{enumerate} 
		\item[(3)] Exceptional graphs $G$ of type $3$ satisfy the following:
		\begin{enumerate}
			\item[(a)] $G_A \neq \emptyset$ and $|G_B| \geq 2$.
			\item[(b)] There exist $1 \leq c_0 \leq n$, a vertex $p_0$ of color $c_0$ and a copy $Q$ of $\subcompgraph^{c_0}$ such that $G=Q \cup \{p_0\}$ and $p_0 \in G_B$.
			\item[(c)] $p_0$ is fully adjacent to $G_A \cup G_C$.
		\end{enumerate} 
	\end{enumerate}
\end{theorem}
\begin{proof}
	\textbf{We show the `if' part.}
%	In any theory, $A \myinddiv_CB$ implies $A \myinddiv_{\myacl(C)}\myacl(BC)$ by \cite[Remark 5.4(3)]{Adler}.
%	Therefore, we may assume that $C=\myacl(C)$ and $B=\myacl(BC)$.
	Let $\overline{a}$ be a tuple of finite length from $A \cup \myacl(C)$.
	Let $\psi(\overline{x},\overline{b})$ be an arbitrary formula such that $\overline{b} \subseteq \myacl(BC)$ and $\mathbb M \models \psi(\overline{a},\overline{b})$.
	Let $(\overline{b}_i\;|\;i<\omega)$ be an indiscernible sequence over $C$ with $\overline{b}_0=\overline{b}$.
	In order to show $A \myinddiv_CB$, we have only to show that $\{\psi(\overline{x},\overline{b}_i)\;|\;i<\omega\}$ is consistent by \cite[Lemma 7.1.4]{TZ}.
	
	Suppose $\overline{a} \cap \myacl(C) \neq \emptyset$.
	Let $l=|\overline{a}|$ and $a_j$ be the $j$-th coordinate of $\overline{a}$.
	We may assume that $a_j \notin \myacl(C)$ if and only if $1 \leq j \leq k$.
	Put $\overline{a'}:=(a_1,\ldots,a_k)$ and $\overline{b'}:=(a_{k+1},\ldots,a_n)$.
	Let $\psi'(\overline{x'},\overline{b'},\overline{b})$ be the formula obtained from $\psi(\overline{x},\overline{b})$ by substituting $a_j$ to the $j$-th coordinate for $k<j \leq n$.
	The sequence $(\overline{b}_i\overline{b'}\;|\; i<\omega)$ is indiscernible over $C$ and, if $\{\psi'(\overline{x'},\overline{b'},\overline{b})\;|\;i<\omega\}$ is consistent, $\{\psi(\overline{x},\overline{b}_i)\;|\;i<\omega\}$ is also consistent.
	Therefore, we may assume that $\overline{a} \cap \myacl(C) = \emptyset$ without loss of generality.
	
	By Proposition \ref{prop:almost_qe}, $\psi$ is a union of finitely many basic existential formulas.
	At least one of these basic formulas belongs to $\mytp(\overline{a}/C\overline{b})$.
	Therefore, we may assume that $\psi$ is a basic existential formula without loss of generality.
	There exists a pair of graphs $(G_1,G_2)$ such that $\psi(\overline{x},\overline{y})= \exists \overline{z}\ \mydiag_{G_1,G_2}(\overline{x}{}^\frown\overline{y},\overline{z})$.
	Since we have $\mathbb M \models \psi(\overline{a},\overline{b})$, there exists $\overline{w} \subseteq \mathbb M$ such that $\mathbb M \models \mydiag_{G_1,G_2}(\overline{a}{}^\frown\overline{b},\overline{w})$.
	
	For every vertex $v \in \overline{w}$, $S_{\overline{a}}(v):=\{u \in \overline{a}\;|\; \mathbb M \models R(u,v)\}$.
	We say that $v \in \overline{w}$ is \textit{tightly adjacent} to $\overline{a}$ if $S_{\overline{a}}(v)$ contains a copy of $\subcompgraph^c$, where $c$ is the color of $v$.
	If $v \in \overline{w}$ is not tightly adjacent to $\overline{a}$, we say that $v$ is \textit{loosely adjacent} to $\overline{a}$.
	By permuting $\overline{w}$ if necessary, $\overline{w}$ is decomposed into three tuples $\overline{d}$, $\overline{e}$ and $\overline{f}$.
	\begin{itemize}
		\item $\overline{d}$ is the collection of vertices in $\overline{w}$ which are algebraic over $C\overline{b}$.
		\item $\overline{e}$ is the collection of vertices in $\overline{w}$ which are tightly adjacent to $\overline{a}$ and non-algebraic over $C\overline{b}$.
		\item $\overline{f}$ is the collection of vertices in $\overline{w}$ which are loosely adjacent to $\overline{a}$ and non-algebraic over $C\overline{b}$.
	\end{itemize}
	Put $\overline{d}_0=\overline{d}$, $\overline{e}_0=\overline{e}$ and $\overline{f}_0=\overline{f}$ for convenience.
	Since $(\overline{b}_i\;|\;i<\omega)$ is indiscernible and $\overline{b}_0=\overline{b}$, for every formula $\eta(\overline{z},\overline{b})$ in $\mytp(\overline{w}/C\overline{b})$, we have $\mathbb M \models \exists \overline{z}\ \eta(\overline{z},\overline{b}_i)$ for $0<i<\omega$.
	Therefore, $\pi_i(\overline{z}):=\{\eta(\overline{z},\overline{b}_i)\;|\; \eta(\overline{z},\overline{b}) \in \mytp(\overline{w}/C\overline{b})\}$ is a type over $C\overline{b}_i$.
	Let $\overline{w}_i$ is a realization of $\pi_i(\overline{z})$.
	Observe that the first $k$ entries in $\overline{w}_i$ are algebraic over $C\overline{b}_i$ and the other entries are not algebraic over $C\overline{b}_i$.
	Since $\mytp(\overline{b}/C)=\mytp(\overline{b}_i/C)$, there exists an automorphism $$\sigma_i:\mathbb M \to \mathbb M$$ fixing $C$ pointwise and $\sigma_i(\overline{b})=\overline{b}_i$ for $i<\omega$.
	Put $\overline{a}_i=\sigma_i(\overline{a})$, $\overline{d}_i=\sigma_i(\overline{d})$, $\overline{e}_i=\sigma_i(\overline{e})$ and $\overline{f}_i=\sigma_i(\overline{f})$.
	Observe that $\sigma_i(\myacl(C\overline{b}))=\myacl(C\overline{b}_i)$.
	
	Let $$X=\myacl\left(C \cup \bigcup_{i<\omega}\overline{b}_i\right).$$
	Observe that $\overline{d}_i$ is contained in $\myacl(C\overline{b}_i)$ for $i<\omega$.
	By the definition of tight adjacency and Proposition \ref{prop:fr_monster_basic}(3), $\overline{e}$ is contained in $\myacl(C\overline{a})$.
	%Therefore, by replacing $\overline{a}$ with $\overline{a}\overline{e}$, we may assume that $\overline{e}$ is empty.
	
	We construct a graph $\mathfrak Y$ by adding new vertices $\overline{a'}$, $\overline{e'}$ and $\overline{f'}_i$  for $i<\omega$ to $X$ by joining vertices by edges in the following manner:
	\begin{enumerate}
		\item[(a)] The formula $\mydiag_{G_1,G_2}(\overline{a'}{}^\frown\overline{b}_i,\overline{d}_i{}^\frown\overline{e'}{}^\frown\overline{f'}_i)$ holds for $i<\omega$.
		\item[(b)] Every vertex in $\overline{a'}$, $\overline{e'}$ and $\overline{f'}_i$ is adjacent to a vertex $v$ in $\myacl(C\overline{b}_i)$ if and only if the counterpart in $\overline{a}$, $\overline{e}$ and $\overline{f}$ is adjacent to $\sigma_i^{-1}(v)$;
		\item[(c)] They are minimally adjacent, that is, every vertex in $\overline{a'}$, $\overline{e'}$ and $\overline{f'}_i$ is adjacent to a vertex in $X \cup \overline{a'} \cup \overline{e'} \cup \bigcup_{j<\omega}\overline{f'}_j$ only when, if they are separated, conditions (a) or (b) do not hold.
	\end{enumerate}
	Observe that every vertex in $\overline{a'}$, $\overline{e'}$ and $\overline{f'}_i$ is adjacent to a vertex $v$ in $\myacl(C\overline{b}_i)$ if and only if the counterpart in $\overline{a}_i$, $\overline{e}_i$ and $\overline{f}_i$ is adjacent to $v$.
	We may assume that $\overline{e'}=\emptyset$ considering $\overline{a'}\overline{e'}$ instead of $\overline{a'}$ and considering $G_1\overline{e'}$ instead of $G_1$.
	
	We show that $\mathfrak Y$ is $\compgraph$-free.
	Assume for contradiction that $\mathfrak Y$ contains a copy $K$ of $\compgraph$.
	We construct exceptional graphs in $\myacl(AC) \cup \myacl(BC)$ under this assumption.
	Put $A_K:=\overline{a'} \cap K \setminus \myacl(C)$, $C_K:=\myacl(C) \cap K$, $F_K:=K \cap \bigcup_{i<\omega}\overline{f'}_i$ and $X_K:=X \cap K \setminus \myacl(C)$.
	Put $I_A:=\mycol(A_K)$, $I_C:=\mycol(C_K)$, $I_F:=\mycol(F_K)$ and $I_X:=\mycol(X_K)$.
	\medskip
	
	\textbf{Claim 1.} The following assertions hold:
	\begin{enumerate}
		\item[(1)] $A_K \neq \emptyset$.
		\item[(2)] $|I_F| \leq 1$. 
		\item[(3)] $X_K$ is not empty and contains a vertex of color $\notin I_F$.
%		\item[(4)] Suppose $I_F \neq \emptyset$. Let $c_F$ be the color of vertices of $F_K$.
%		There exists $k<\omega$ such that $A_K \cup (X_K \cap \myacl(C\overline{b}_k))$ contains a copy of $\subcompgraph^{c_F}$, and $X_K$ contains a vertex of color $c_F$.
	\end{enumerate}
	\begin{proof}[Proof of Claim 1]
		(1) Suppose not.
		$K$ is contained in the graph $X \cup \bigcup_{i<\omega}\overline{f'}_i$, which is isomorphic to the graph constructed from $F:=X \cup  \bigcup_{j<\omega}\overline{f}_j$ by removing some edges.
		Therefore, $F$ also contains a copy of $\compgraph$.
		However, $F$ is contained in the $\compgraph$-free graph $\mathbb M$, which is absurd.
		
		(2) Suppose $F_K \neq \emptyset$.
		Put $H_i:=F_K \cap \overline{f'}_i$.
		Suppose that $F_K$ contains vertices of multiple colors.
		Since $H_i$ is fully separated from $H_j$ if $i \neq j$, there exists $k<\omega$ such that $F_K \subseteq \overline{f'}_{k}$.
		Since there are no edges joining $\overline{f'}_{k}$ with $X \setminus \myacl(Cb_k)$, $K \cap X = K \cap \myacl(Cb_k)$.
		This implies that $F':=\myacl(Cb_k) \cup \overline{a}_k \cup \overline{f_k}$ contains a copy of $\compgraph$, which is a contradiction because $F'$ is a substructure of $\mathbb M$.
		Therefore, $|I_F|=1$.
		
		(3) Suppose $X_K$ is empty or only contains vertices of color $c_F \in I_F$.
		If $X_K=\emptyset$, $A_K \cup C_K$ contains a copy of $\compgraph$ or contains a copy of $\subcompgraph^{c_F}$, depending on whether $F_K=\emptyset$ or not.
		If $I_X=I_F$, $A_K \cup C_K$ contains a copy of $\subcompgraph^{c_F}$.
		This implies that $F_K$ is contained in $\myacl(C\overline{a'})$ by Proposition \ref{prop:fr_monster_basic}(3).
		Therefore, the vertex in $\overline{f}$ corresponding to a vertex in $F_K$ is contained in $\myacl(C\overline{a})$ of $\mathbb M$, which contradicts the definition of $\overline{f}$.
%		
%		(4) By (2), there exists a vertex $v \in X_K$ of color $\neq c_F$.
%		$v$ is adjacent to vertices in $F_K$.
%		By the definition of $\mathfrak Y$, $v$ is contained in $\myacl(C\overline{b}_k)$ and $F_K$ is contained in $\overline{f'}_k$.
%		Every vertex in $K$ of color $\neq c_F$ is adjacent to $F_K$.
%		Therefore, every vertex in $X_K$ of color $\neq c_F$ is contained in $X_K \cap \myacl(C\overline{b}_k)$.
%		This implies that $A_K \cup (X_K \cap \myacl(C\overline{b}_k))$ contains a copy of $\subcompgraph^{c_F}$.
%		
%		Assume for contradiction that $X_K$ does not contain a vertex of color $c_F$.
%		$K$ is contained in $\overline{a'} \cup \overline{f'}_k \cup \myacl(C\overline{b}_k)$.
%		This implies that the subgraph $\overline{a}_k \cup \overline{f}_k \cup \myacl(C\overline{b}_k)$ of $\mathbb M$ contains a copy of $\compgraph$, which is absurd.
	\end{proof}
	
	We consider three separate cases.
	The first case is the case where $F_K=\emptyset$ and $A_K$ is not monochromatic.
	\medskip 
	
	\textbf{Claim 2.} Suppose $|I_A|>1$. For every vertex $v$ in $X_K$, there exist $i_v<\omega$ such that  $v \in \myacl(C\overline{b}_{i_v})$.
	\begin{proof}[Proof of Claim 2]
		$v$ is adjacent to a vertex in $A_K$ because $|I_A|>1$.
		By the definition of $\mathfrak Y$, $\overline{a'}$ is fully separated from $X \setminus \bigcup_{i<\omega}\myacl(C\overline{b}_i)$.
		Therefore, there exists $i_v<\omega$ with $v \in \myacl(C\overline{b}_{i_v})$.
	\end{proof}

	Assume $F_K= \emptyset$ and $|I_A|>1$, then we have $K=A_K \cup C_K \cup X_K$ in this case.
	We define a map $\mathfrak p: K \to \myacl(AC) \cup \myacl(BC)$.
	For $v \in A_K$, there exists a counterpart of $v$ in $\overline{a}$.
	We define $\mathfrak p(v)$ as this counterpart.
	For $v \in X_K$, there exists $i_v$ such that $c \in \myacl(C\overline{b}_{i_v})$ by Claim 2.
	We put $\mathfrak p(v):=\sigma_{i_v}^{-1}(v) \in \myacl(C\overline{b})$.
	Put $\mathfrak p(v)=v$ for $v \in C_K$ and $K_{\mathfrak p}:=\mathfrak p(K)$.
	
	Let $A_K[c]:=\{v \in A_K\;|\; \mycol(v)=c\}$ for $1 \leq c \leq n$.
	We define $C_K[c]$, $X_K[c]$ and $K_{\mathfrak p}[c]$ in the same manner.
	\medskip
	
	\textbf{Claim 3.} Suppose $|I_A|>1$ and $F_K=\emptyset$. The following assertions hold:
	\begin{enumerate}
		\item[(1)] $\mathfrak p(A_K) \neq \emptyset$ and $\mathfrak p(X_K) \neq \emptyset$.
		The inequality $1 \leq |K_{\mathfrak p}[c]| \leq m_c$ holds for every $1 \leq c \leq n$, and
		$|K_{\mathfrak p}[c]| =m_c$ holds for $c \notin \mycol(X_K)$.
		\item[(2)] Every $v \in \mathfrak p(A_K \cup C_K)$ is fully adjacent to $K_{\mathfrak p}$.
		\item[(3)] Put $k_c:=m_c-|A_K[c]|-|C_K[c]|$ for $1 \leq c \leq n$ and $\overline{k}=(k_1,\ldots,k_n)$.
		Then $\mathfrak p(X_K)$ has a nice $\overline{k}$ cover over $C$.
		\item[(4)] $\myacl(AC) \cup \myacl(BC)$ contains an exceptional graph of type $1$.
	\end{enumerate}
	\begin{proof}[Proof of Claim 3]
		(1) We have $\mathfrak p(A_K) \neq \emptyset$ and $\mathfrak p(X_K) \neq \emptyset$ by Claim 1(1,3).
		The inequalities and equalities in (1) are obvious from the definition of $K_{\mathfrak p}$. 
		
		(2) By the definition of $\mathfrak Y$, it is obvious that $v_1$ is adjacent to $v_2$ if $v_1,v_2 \in \mathfrak p(A_K \cup C_K)$.
		Consider the case where $v_1 \in \mathfrak p(X_K)$ and $v_2 \in \mathfrak p(C_K)$.
		There exist $j<\omega$ and $v \in X_K \cap \myacl(C\overline{b}_j)$ such that $v=\sigma_j(v_1)$.
		$v$ is adjacent to $v_2$.
		$v_1=\sigma_j^{-1}(v)$ is adjacent to $v_2=\sigma_j^{-1}(v_2)$.
		The remaining case is the case where $v_1 \in \mathfrak p(X_K)$ and $v_2 \in \mathfrak p(A_K)$. 
		Pick $j$ and $v$ in the same manner as the previous case.
		By the definitions of $\mathfrak Y$ and $\mathfrak p$, $\sigma_j(v_2)$ is adjacent to $v$.
		Therefore, $v_2$ is adjacent to $v_1=\sigma_j^{-1}(v)$.
	
		(3) Put $X_K\langle i \rangle:=X_K \cap \myacl(C\overline{b}_{i})$ for $i<\omega$.
		Put $\mathfrak q_j(v):=\sigma_j(v) \in \myacl(C\overline{b}_j)$ for $j<\omega $ and $v \in K_{\mathfrak p} \cap \myacl(C\overline{b}) \setminus \myacl(C)=\mathfrak p(X_K)$.
		There exists $j_1,\ldots, j_l<\omega$ such that $X_K\langle j_i \rangle:=X_K \cap \myacl(C\overline{b}_{j_i}) \neq \emptyset$ and $X_K=\sum_{i=1}^l X_K\langle j_i \rangle$.
		Since $X_K\langle j_i \rangle$ is a subgraph of the complete graph $K$, it is a complete graph.
		Put $U_i:=\mathfrak p(X_K\langle j_i \rangle)$ for $1 \leq i \leq l$.
		
		We show that these $U_i$ satisfy conditions (i) through (iii) in Definition \ref{def:fine_cover}.
		The restriction of $\mathfrak p$ to $X_K\langle j_i \rangle$ is the restriction of the isomorphism $\sigma_{j_i}^{-1}$, and $U_i$ is also a complete graph.
		It is obvious that $\mathfrak p(X_K)=\bigcup_{i=1}^l U_i$. 
		They demonstrate that $\{U_i\}_{1 \leq i \leq l}$ satisfies condition (i).
		
		Put $\mathfrak n_{ij}:=|\{v \in U_i\;|\; \mycol(v)=j\}|$, then $\mathfrak n_{ij}=|X_K[j] \cap \myacl(C\overline{b}_{j_i})|$.
		Therefore, $\sum_{i=1}^l \mathfrak n_{ij} = |X_K[j]|=m_j-|A_K[j]|-|C_K[j]|=k_j$ for $1 \leq j \leq n$.
		We have shown that $\{U_i\}$ satisfies condition (ii).
		
		The remaining task is to show that $\{U_i\}$ satisfies condition (iii).
		Let us consider the subgraph $S_i$ of $\mathbb M$ consisting of $\myacl(C) \cup \bigcup_{j \leq i} \mathfrak q_j (\mathfrak p(X_K))$.
		Let $U_{ij}:=\mathfrak q_i(U_j)$.
		$X_K\langle j_i \rangle$ is fully adjacent to $X_K\langle j_{i'} \rangle$ if $i<i'$ because they are subgraphs of the complete graph $K$.
		By indiscerniblity of $(\overline{b}_i\;|\;i<\omega)$ over $C$, $U_{ij}$ is fully adjacent to $U_{i'j'}$ for every $i<i'$ and $1 \leq j<j' \leq l$.
		Therefore, $S_i$ is isomorphic to the graph constructed from $N_i$ in Definition \ref{def:fine_cover} possibly by joining some vertices by extra edges.
		Since $S_i$ is $\compgraph$-free, $N_i$ is $\compgraph$-free as well.
		
		(4) By (1,2,3), $K_{\mathfrak p}$ is an exceptional graph of type $1$ in $\myacl(AC) \cup \myacl(BC)$.
	\end{proof}

	Next, we consider the case where either $|I_A|=1$ and $F_K=\emptyset$ or $|I_A|=1$ and $I_A=I_F$.
	In both cases, $I_F \subseteq I_A$ and $A_K$ is monochromatic.
	We can prove the following claim in the same manner as Claim 2.
	We omit the proof.
	\medskip 
		
	\textbf{Claim 4.} Suppose $|I_A|=1$. Let $c_A$ be the unique color of vertices in $A_K$. 
	For every vertex $v$ in $X_K$ of color $\neq c_A$, there exist $i_v<\omega$ such that  $v \in \myacl(C\overline{b}_{i_v})$.	
	\medskip
	
	Suppose $I_F \subseteq I_A$ and $|I_A|=1$.
	Let $I_A=\{c_A\}$. 
	Let $K':=\{v \in K\;|\; \mycol(v) \neq c_A\} \cup \{v \in A_K \cup C_K\;|\;\mycol(v)=c_A\}$.
	$K'$ is contained in $A_K \cup X_K \cup C_K$. % if $F_X=\emptyset$ and $A_X$ is monochromatic.
	We can define $\mathfrak p:K' \to \myacl(AC) \cup \myacl(BC)$ in the same manner as the case where $F_K=\emptyset$ and $|I_A|>1$ using Claim 4 instead of Claim 2.
	We also define $K'_{\mathfrak p}$ and $K'_{\mathfrak p}[c]$ in the same manner as $K_{\mathfrak p}$ and $K_{\mathfrak p}[c]$.
	\medskip
	
	\textbf{Claim 5.} Suppose $|I_A|=1$ and $I_F \subseteq I_A$. Let $I_A=\{c_A\}$. The following assertions hold:
	\begin{enumerate}
		\item[(1)] $\mathfrak p(A_K) \neq \emptyset$. $\mathfrak p(X_K)$ has a vertex of color $\neq c_A$.
		The inequality $1 \leq |K'_{\mathfrak p}[c]| \leq m_c$ holds for every $1 \leq c \leq n$, and
		$|K'_{\mathfrak p}[c]| =m_c$ holds for $c \notin \mycol(X_K)$ and $c \neq c_A$.
		\item[(2)] Every $v \in \mathfrak p(A_K \cup C_K)$ is fully adjacent to $K'_{\mathfrak p}$.
		\item[(3)] Put $k_{c_A}=0$ and $k_c=m_c-|C_K[c]|$ for $1 \leq c \leq n$ with $c \neq c_A$ and $\overline{k}=(k_1,\ldots,k_n)$.
		Then $\mathfrak p(X_K \setminus X_K[c_A])$ has a nice $\overline{k}$ cover over $C$.
		\item[(4)] We may assume that there are no $v \in \myacl(C)$ of color $c_A$ fully adjacent to $K'_{\mathfrak p}$ other than the vertices in $\mathfrak p(C_K)$.
%		\item[(4)] $K'_{\mathfrak p}$ does not contain a copy of $\subcompgraph^{c_A}$.
%		\item[(5)] Put  $R_{K'}:=\{v \in \myacl(C)\;|\; \mycol(v)=c_A, v \text{ is fully adjacent to } \mathfrak p(X_K) \cup \mathfrak p(C_K)\}$. Then, $|R_{K'}| <m_{c_A}$.
		\item[(5)] $K'_{\mathfrak p}$ is an exceptional graph of type $1$ or $2$.
	\end{enumerate}
	\begin{proof}[Proof of Claim 5]
		(1) We can prove (1) in the same manner as Claim 3(1) except the claim that $\mathfrak p(X_K)$ has a vertex of color $\neq c_A$.
		Assume for contradiction that $\mathfrak p(X_K)$ does not have a vertex of color $\neq c_A$.
		By Claim 1(3), $X_K$ has at least one vertex $p$ of color $c_A$.
		Since $I_A=\{c_A\}$, $C_K$ contains a copy $K''$ of $\subcompgraph^{c_A}$.
		By Proposition \ref{prop:fr_monster_basic}(3), $\myacl(C)$ contains a tuple of $m_{c_A}-1$ vertices $\overline{w}$ of color $c_A$ fully adjacent to $K''$.
		The graph $K''\overline{w}p$ contained in $\mathbb M$ is a copy of $\compgraph$, which is absurd.
		
		(2,3) The proofs of for (2) and (3) are almost the same as those for the counterparts of Claim 3.
		We omit the details.
%		
%		(4) If $K'_{\mathfrak p}$ contains a copy of $\subcompgraph^{c_A}$, $\mathfrak p(A_K)$ is contained in $\myacl(AC) \cap \myacl(BC)=\myacl(C)$, which contradicts the definition of $A_C$.
		
%		(5) Assume for contradiction that there exists a tuple of vertices $\overline{v}$ of color $c_A$ of length $m_{c_A}$ in $\myacl(C)$ fully adjacent to $\mathfrak p(X_K) \cup \mathfrak p(C_K)$.
%		By the definition of $\mathfrak p$, $\overline{v}$ is fully adjacent to $X_K \cup C_K$.
%		However, $X_K \cup C_K$ has a copy of $\subcompgraph^{c_A}$.
%		This implies that the subgraph $\overline{v} \cup X_K \cup C_K$ of $\mathbb M$ contains a copy of $\compgraph$, which is absurd.
%		
		(4) Put $T:=\{v \in \myacl(C)\;|\; \mycol(v)=c_A, \ v \text{ is fully adjacent to }K'_{\mathfrak p}\}$.
		Since $(\overline{b}_i\;|\;i<\omega)$ is indiscernible, by Claim 4, $T$ is fully adjacent to $\{v \in K\;|\; \mycol(v)\neq c_A\}$, which is a copy of $\subcompgraph^{c_A}$.
		If $|T| \geq m_{c_A}$, $\mathbb M$ contains a copy of $\compgraph$, which is absurd.
		Therefore, we have $|T|<m_{c_A}$.
		By replacing $K$ so that $T$ is contained in $C_K$, we may assume that there are no $v \in \myacl(C)$ of color $c_A$ fully adjacent to $K'_{\mathfrak p}$ other than the vertices in $\mathfrak p(C_K)$.

		(5) By (1-4), $K'_{\mathfrak p}$ is an exceptional graph of type $1$ if $|K'_{\mathfrak p}[c_A]|=m_{c_A}$, and it is an exceptional graph of type $2$ otherwise.
	\end{proof}

	Recall that $F_K$ is monochromatic by Claim 1(2).
	The remaining case is the case where either $|I_A|=|I_F|=1$ and $I_A \neq I_F$ or $F_K \neq \emptyset$ and $|I_A|>1$.
	In both cases, we have $I_A \not\subseteq I_F$ and $F_K \neq \emptyset$. 
	\medskip

	\textbf{Claim 6.} Suppose $I_A \not\subseteq I_F$ and $F_K \neq \emptyset$. Let $c_F$ be the unique color contained in $I_F$. Then, the following assertions hold:
	\begin{enumerate}
		\item[(1)] There exists $k_F <\omega$ and $l_X<\omega$ such that $k_F \neq l_X$, $F_K \subseteq \overline{f'}_{k_F}$, $X_K \cap \myacl(\overline{b}_{l_X})$ contains a vertex of color $c_F$.
		In addition, $A_K \cup C_K \cup (X_K \cap \myacl(\overline{b}_{k_F}))$ contains a copy of $\subcompgraph^{c_F}$.
		\item[(2)] $\myacl(AC)\cup \myacl(BC)$ contains an exceptional graph of type $3$. 
	\end{enumerate}
	\begin{proof}[Proof of Claim 4]
		(1)  Choose $k_F<\omega$ so that $\overline{f'}_{k_F} \cap F_K \neq \emptyset$.
		$X_K$ contains a vertex $v$ of color $\neq c_F$ by Claim 1(3).
		Since $\overline{f'}_{k_F}$ is fully separated from $X \setminus \myacl(C\overline{b}_{k_F})$ by the definition of $\mathfrak Y$, $X_K \cap \myacl(C\overline{b}_{k_F})$ contains a vertex $v$ of color $\neq c_F$, and $F_K$ is contained in $\overline{f'}_{k_F}$.
		
		Put $K':=\{v \in K\;|\; \mycol(v) \neq c_F\}$.
		It is contained in $A_K \cup C_K \cup X_K$.
		We have $K' \cap X_K \cap \myacl(C\overline{b}_l)=\emptyset$ for $l \neq k_F$ because there are no edges between $\overline{f'}_{k_F}$ and $\myacl(C\overline{b}_l)$ by the definition of $\mathfrak Y$.
		Therefore, $K' \subseteq A_K \cup C_K \cup \myacl(C\overline{b}_{k_F})$.
		We have finished the proof of the `in addition' part.
		
		If $F_K \cup A_K \cup C_K$ contains $m_{c_F}$ vertices of color $c_F$, the subgraph $\overline{a}_{k_F} \cup C_K \cup \myacl(C\overline{b}_{k_F}) \cup \overline{f}_{k_F}$ of $\mathbb M$ contains a copy of $\compgraph$, which is absurd.
		Therefore, there exists a vertex $u$ of color $c_F$ in $X_K$.
		If such vertices are contained in $\myacl(C\overline{b}_{k_F})$, $\overline{a}_{k_F} \cup C_K \cup \myacl(C\overline{b}_{k_F}) \cup \overline{f}_{k_F}$ contains a copy of $\compgraph$, which is absurd.
		The vertex $u$ is contained in $X \setminus \myacl(C\overline{b}_{k_F})$, and fully adjacent to a vertex in $A_K$ of color $\neq c_F$, which exists because $I_A \not\subseteq I_F$.
		By the definition of $\mathfrak Y$, there exists $l_X$ such that $u \in \myacl(C\overline{b}_{l_X})$.
		
		(2) By (1), $A_K \cup C_K \cup \myacl(C\overline{b}_{k_F})$ contains a copy of $\subcompgraph^{c_F}$.
		Therefore, $K'_{k_F}:=\overline{a}_{k_F} \cup \myacl(C) \cup \myacl(C\overline{b}_{k_F})$ contains a copy of $\subcompgraph^{c_F}$ by the definition of $\mathfrak Y$.
		The graph $\sigma_{k_F}^{-1}(K'_{k_F}) \subseteq \myacl(AC) \cup \myacl(BC)$ also contains a copy $Q$ of $\subcompgraph^{c_F}$.
		
		Let $p$ be a vertex in $X_K \cap \myacl(C\overline{b}_{l_X})$ of color $c_F$.
		$p$ is fully adjacent to $\sigma_{l_X}(Q \cap \myacl(C\overline{a}))$ by the definition of $\mathfrak Y$, $A_K$ and $C_K$.
		$p_0:=\sigma_{l_X}^{-1}(p)$ is fully adjacent to $Q \cap \myacl(C\overline{a})$.
		The graph consisting of $Q$ and $p_0$ is an exceptional graph of type $3$.
	\end{proof}

	We obtained contradictions in Claim 3(4), Claim 5(5) and Claim 6(2).
	We have shown that $\mathfrak Y$ is $\compgraph$-free.
	By Proposition \ref{prop:fr_monster_basic}(1), there exists an embedding $\iota:\mathfrak Y \hookrightarrow \mathbb M$ fixing $X$ pointwise.
	The image $\iota(\overline{a'})$ is a solution of $\{\psi(\overline{x},\overline{b}_i)\;|\;i<\omega\}$.	
	\medskip
	
	\textbf{We show the contraposition of the `only if' part.}
	$A \mynindalg_CB$ implies $A \myninddiv_CB$ by \cite[Reamrk 5.4(4)]{Adler}.
	We assume $A \myindalg_CB$ in the rest of this proof.
	
	We consider the cases where $\myacl(AC) \cup \myacl(BC)$ contains a copy $G$ of the exceptional graph of type $1$, $2$ and $3$.
	By \cite[Lemma 7.1.4]{TZ}, we have $A\myninddiv_{\myacl(C)}B \Rightarrow A\myninddiv_CB$.
	By the definition of exceptional graphs, it is obvious that every exceptional graph with respect to $C$ is an exceptional graph with respect to $\myacl(C)$.
	Therefore, we may assume that $C=\myacl(C)$. 
	We may assume that $B=\myacl(BC)$ because $A \myninddiv_{C}\myacl(BC)$ implies $A \myninddiv_CB$ in any theory by \cite[Remark 5.4(3)]{Adler}.
	
	In order to show $A \myninddiv_CB$, we construct a positive integer $k_G$ and an $\colorgL(C)$-formula $\phi_G(\overline{x},\overline{b})$ with parameters $\overline{b} \in B$ and a sequence $(\overline{b}_i\;|\;i<\omega)$ such that $\overline{b}_0=\overline{b}$, $\mytp(\overline{b}_i)=\mytp(\overline{b})$ for $i<\omega$ and $\{\phi_G(\overline{x},\overline{b})\;|\;i<\omega\}$ is $k_G$-inconsistent.
	
	Let us suppose that $\myacl(AC) \cup B$ contains a copy $G$ of an exceptional graph of type $1$.
	Let $\overline{a}$ and $\overline{b}$ be the enumeration of elements $G_A \cup G_C$ and $G_B$, respectively. 
	Put $k_c=m_c-|\{v \in G_A \cup G_C\;|\;\mycol(v)=c\}|$ for $1 \leq c \leq n$ and $\overline{k}=(k_1,\ldots,k_n)$.
	Since $G_B$ has a nice $\overline{k}$-cover over $C$, there exist subgraphs $U_1,\ldots, U_l$ satisfying conditions (i) through (iii) in Definition \ref{def:fine_cover}.
	
	Let $M$ be a small model of $\randfrcolgT$ containing $B$.
	Let $M_i$ be copies of $M$ with $M_0=M$ for $i<\omega$.
	Let $\overline{b}_i$ be the enumeration of a copy of $G_B$ in $M_i$.
	The copy of $U_j$ in $M_i$ is denoted by $U_{ij}$. 
	We construct graphs $N_i$ for $i<\omega$ so that $M_{i'}$ is a subgraph of $N_i$ for $i' \leq i$  as follows:
	Put $N_0:= M_0$.
	For $i>0$, $N_i$ be the graph constructed from $N_{i-1} \otimes_{\myacl(C)} M_i$ by joining vertices in $M_i$ with vertices in $N_{i-1}$ by edges so that, for every $i'<i$, $U_{ij}$ is fully adjacent to $U_{i'j'}$ if $1 \leq j'<j \leq l$, and $U_{ij}$ is fully separated to $U_{i'j'}$ if $1 \leq j<j' \leq l$. 
	Then, by condition (iii) of Definition \ref{def:fine_cover}, $N_i$ is $\compgraph$-free for all $i<\omega$.
	The graph $N:=\bigcup_{i<\omega}N_i$ is also $\compgraph$-free.
	By Proposition \ref{prop:fr_monster_basic}(1), we can embed $N$ into $\mathbb M$ fixing $M$ pointwise.
	In addition, $M_i$ is a model of $\randfrcolgT$ and contained in $\mathbb M$.
	Since $M_i$ is an elementary substructure of $\mathbb M$ by Lemma \ref{lem:fr_model_complete}, we can extend the canonical isomorphism $f_i;M_i \to M_0$ to an automorphism in $\mathbb M$.
	Therefore, we have $\mytp(\overline{b}_i/C)=\mytp(\overline{b}_0/C)=\mytp(\overline{b}/C)$.
	
	Put $k_G:=l$. Let $H:=G_A \cup G_C$.
	Let $\overline{a}_1$ be the enumeration of $G_A$ and $\overline{a}_2$ be the enumeration of $G_C$.
	We may assume that $\overline{a}=\overline{a}_1{}^\frown\overline{a}_2$.
%	We can choose an $\colorgL(C)$-formula $\chi(\overline{t},\overline{x})$ with $|\overline{t}|=|\overline{a}|$ such that it is a conjunction of atomic formulas and $\mathbb M \models \chi(\overline{t},\overline{b}) \Leftrightarrow \overline{t}=\overline{a}$.
	%Let $$\phi_G(\overline{x},\overline{y}):= \exists \overline{t}\ \chi(\overline{t},\overline{x}) \wedge  \mydiag_{H,G}(\overline{t},\overline{y}).$$ 
	Let $\phi_G(\overline{x},\overline{y}):= \mydiag_{H,G}(\overline{x}{}^\frown\overline{a}_2,\overline{y}).$ 
	We have $\mathbb M \models \phi_G(\overline{a}_1,\overline{b}).$
	We show that $(\phi_G(\overline{x};\overline{b}_i)\;|\;i<\omega)$ is $k_G$-inconsistent.
	Let $i_1<i_2<\ldots<i_{k_G}<\omega$.
	Assume for contradiction that $\{\phi_G(\overline{x};\overline{b}_{i_j})\;|\;1 \leq j \leq k_G\}$ has a solution $\overline{a'}$.
	Observe that $U_{i_j,k} \subseteq \overline{b}_{i_j}$ for $1 \leq k \leq l$.
	The graph $\overline{a'} \cup \overline{a}_2 \cup \bigcup_{j=1}^{k_G} U_{i_j,j}$ is a copy of $\compgraph$, which is absurd.
	We have shown $A \myninddiv_CB$ in this case.
	\medskip
	
	Next we treat the case where there exists an exceptional graph $G$ of type $2$ in $\myacl(AC) \cup B$.
	We may assume that the vertexes in $G_A$ are of color $1$ without loss of generality.
	Let $\overline{b}$ be the enumeration of elements in $G_B$.
	We define $M_i$, $N$, $\overline{b}_i$, $U_{ij}$ in the same manner as the case of exceptional graph of type $1$.
	$N$ is $K_m$-free.
	Put $c_0:=|\{v \in G_C\;|\;\mycol(v)=1\}|$ and $k_0=m_1-1-c_0>0$.
	Observe that $|G_A| \leq k_0$.
	For every $\overline{i}:=(i_1,i_2,\ldots,i_l)$ with $i_1<\ldots<i_l<\omega$, add new tuple $\overline{p}\langle{\overline{i}}\rangle$ of vertices of color $1$ of length $k_0$.
	Join $\overline{p}\langle{\overline{i}}\rangle$ with vertexes by edges so that $\overline{p}\langle {\overline{i}}\rangle$ is fully adjacent to $U_{i_jj}$ for $1 \leq j \leq l$.
	Let $N'$ be the graph constructed by adding new vertices as above.
	Since $|G[1]|<m_1$, and $G$ satisfies condition (d) in the definition of exceptional graphs of type $2$, $N'$ is $\compgraph$-free. 
	By Proposition \ref{prop:fr_monster_basic}(1), we can embed $N'$ into $\mathbb M$ fixing $M$ pointwise.
	In the same manner as the case of exceptional graphs of type $1$, we have $\mytp(\overline{b}_i/C)=\mytp(\overline{b}_0/C)=\mytp(\overline{b}/C)$.
	
	Put $k_G=l+1$. Let $\overline{x}$ be a tuple of variables of length $|G_A|$. Let $\phi_G(\overline{x},\overline{y})$ be the formula with parameters from $C$ saying that $\overline{x}$ are of color $1$ and $\overline{x}\overline{y}G_C$ is isomorphic to $G$.
	It is obvious that $\mathbb M \models \phi_G(\overline{a},\overline{b})$.
	We show that $\{\phi_G(\overline{x},\overline{b}_i)\;|\;i<\omega\}$ is $k_G$-inconsistent.
	Let $i_1<i_2<\ldots<i_{k_G}<\omega$ be an arbitrary increasing sequence.
	Put $\overline{i}_1=(i_1,\ldots,i_l)$ and $\overline{i}_2=(i_2,\ldots,i_{l+1})$.
	Since $\bigcup_{j=k}^{l+k-1}U_{i_jj}$ contains a copy of $\subcompgraph^1$, every solution of $\{\phi_G(\overline{x},\overline{b}_i)\;|\; i \in \overline{i}_k\}$ for $k=1,2$ is a subtuple of the vertices of color $1$ of length $m_1-1$ fully adjacent to $\bigcup_{j=k}^{l+k-1}U_{i_jj}$.
	By the definition of $N'$, the solutions of $\{\phi_G(\overline{x},\overline{b}_i)\;|\; i \in \overline{i}_k\}$ are subtuples of $\overline{p}\langle\overline{i}_k\rangle$ for $k=1,2$.
	Therefore, $\{\phi_G(\overline{x},\overline{b}_{i_j})\;|\;1 \leq j \leq k_G\}$ has no common solution.
	\medskip
	
	The remaining case is the case where there exists an exceptional graph $G$ of type $3$ in $\myacl(AC) \cup B$.
	Let $c_0$ be the color given in the definition of exceptional graphs of type $3$.
	We may assume that $c_0=1$ without loss of generality. 
	Let $M_0$ be a small elementary substructure of $\mathbb M$ containing $G$.
	Let $M_i$ be copies of $M_0$ for $i<\omega$.
	Let $G_{B,i}$ and $p_{0,i}$ be the copies of $G_B$ and $p_0$ in $M_i$ for $i<\omega$, respectively.
	Put $\overline{b}_i:=G_{B,i} \cup \{p_{0,i}\}$.
	
	We construct a $\compgraph$-free graph $N_i$ for $i<\omega$ having $M_j$ $(j \leq i)$ as a subgraph by induction on $i$ as follows: 
	Put $N_0:=M_0$.
	Let $N'_i := N_{i-1} \otimes_{C}M_i$.
	$N'_i$ is $\compgraph$-free by Lemma \ref{lem:amalgam_free_colored}.
	We join vertices in $G_{B,i}$ with $\overline{p}_{0,j}$ by edges if $i \neq j$.
	Let $N_i$ be the graph constructed from $N'_i$ by adding the edges as above.
	We show that $N_i$ is $\compgraph$-free.
	Assume for contradiction that $N_i$ contains a copy $K$ of $\compgraph$.
	Any vertex of color $\neq 1$ in $M_j \setminus C$ is separated from all the vertex of color $\neq 1$ in $M_{j'} \setminus C$ for $j,j' \leq i$ with $j \neq j'$.
	Therefore, the subgraph $K'$ of $K$ which is isomorphic to $\subcompgraph^1$ is contained in $M_j$ for some $j \leq i$.
	We assume that $j<i$, but we can lead to a contradiction in the same manner if $j=i$.
	Since $N_{i-1}$ is $\compgraph$-free by the induction hypothesis, $K$ contains a vertex in $M_i \setminus C$.
	$p_{0,i}$ is the unique vertex adjacent to some vertex in $M_j \setminus C$.
	Therefore, $p_{0,i}$ is contained in $K$.
	However, $p_{0,i}$ is only adjacent to the vertices in $G_{B,j}$, which is not a copy of $\subcompgraph^1$.
	This is a contradiction.
	Put $N=\bigcup_{i<\omega}N_i$, which is also $\compgraph$-free.
	
	By Proposition \ref{prop:fr_monster_basic}(1), we can embed $N$ into $\mathbb M$ fixing $M_0$ pointwise.
	Therefore, we may assume that $\overline{b}_i$ is contained in $\mathbb M$.
	We have $\mytp(\overline{b}_i)=\mytp(\overline{b})$ for $i<\omega$ for the same reason as the case of type $1$.
	Let $\overline{a}_1$ be the enumeration of $G_A$ and $\overline{a}_2$ be the enumeration of $G_C$.
%	We can choose an $\colorgL(C)$-formula $\chi(\overline{t},\overline{x})$ with $|\overline{t}|=|\overline{a}|$ such that it is a conjunction of atomic formulas and $\mathbb M \models \chi(\overline{t},\overline{b}) \Leftrightarrow \overline{t}=\overline{a}$.
	Put $k_G=m_1$ and let $\phi_G(\overline{x},\overline{y}):=  \mydiag_{G_A,G}(\overline{x},\overline{a}_2{}^\frown\overline{y}).$
	We show that $\{\phi_G(\overline{x},\overline{b}_i)\;|\;i<\omega\}$ is $k_G$-inconsistent.
	Let $i_1<i_2<\ldots<i_{k_G}<\omega$.
	Assume for contradiction that there exists a common realization $\overline{a}$ of $\{\phi_G(\overline{x},\overline{b}_{i_j})\;|\;1 \leq j \leq k_G\}$.
	$\overline{a} \cup G_{B,i_1}$ is a copy of $\subcompgraph^1$ and it is fully adjacent to $m_1$ distinct vertices $p_{0,i_1},\ldots, p_{0,i_{k_G}}$ of color $1$.
	This means that $\mathbb M$ contains a copy of $\compgraph$, which is absurd.
\end{proof}

\begin{corollary}\label{cor:divide_fork}
	If $n>2$, dividing independence $\myinddiv$ coincides with forking independence $\myindfork$ in $\randfrcolgT$.
\end{corollary}
\begin{proof}
	Let $A,B$ and $C$ be a small subset of $\mathbb M$.
	The implication ${A}\myindfork_{C}{B} \Rightarrow {A}\myinddiv_{C}{B}$ is well-known.
	We show the opposite implication.
	Suppose ${A}\myinddiv_{C}{B}$.
	Let $\widehat{B}$ be a small subset of $\mathbb M$ containing $B$.
	We have only to find a small subset $A' \subseteq \mathbb M$ such that $A' \equiv_{BC}A$ and $A' \myinddiv_C\widehat{B}$.
	
	Let $M$ be a submodel of $\randfrcolgT$ contained in $\mathbb M$ and containing $A\widehat{B}C$.
	Let $M'$ be a copy of $M$.
	Put $N:=M \otimes_{\myacl(BC)} M'$, which is $\compgraph$-free by Lemma \ref{lem:amalgam_free_colored} and Proposition \ref{prop:fr_monster_basic}(3).
	By Proposition \ref{prop:fr_monster_basic}(1), we can embed $N$ into $\mathbb M$ fixing $M$ pointwise.
	Therefore, we may assume that $N$ is a substructure of $\mathbb M$.
	Let $A'$ be the copy of $A$ in $M'$.
	By the definition of free amalgam, the relation $A' \equiv_{BC}A$ obviously holds.
	
	We show $A' \myinddiv_C\widehat{B}$.
	We use the equivalent condition given in Theorem \ref{thm:dividing}.
	By the definition of free amalgam, we have $\myacl(A'C) \cap \myacl(\widehat{B}C) \subseteq \myacl(BC)$.
	Since ${A}\myinddiv_{C}{B} \Rightarrow {A}\myindalg_{C}{B}$ holds in any theory, we have $\myacl(A'C) \cap \myacl(BC)=\myacl(C)$.
	Therefore, we get $\myacl(A'C) \cap \myacl(\widehat{B}C)=\myacl(A'C) \cap  \myacl(BC)=\myacl(C)$.
	
	Assume for contradiction that $\myacl(A'C) \cup \myacl(\widehat{B}C)$ contains an exceptional graph $G$ of type $1$, $2$ or $3$ defined in Theorem \ref{thm:dividing}.
	We put $G_{A'}:=G \cap \myacl(A'C) \setminus \myacl(C)$, $G_{\widehat{B}}:=G \cap \myacl(\widehat{B}C) \setminus \myacl(C)$ and $G_C=G \cap \myacl(C)$.
	Since $\myacl(A'C) \cap \myacl(\widehat{B}C)=\myacl(C)$ and $G_{A'}$ is contained in $M'$, $G_{A'}$ is contained in $M' \setminus \myacl(BC)$.
	Whatever the type of $G$ is, $G_{A'}$ is fully adjacent to $G_{\widehat{B}}$.
	On the other hand, $M \setminus \myacl(BC)$ is fully separated from $M' \setminus \myacl(BC)$.
	Therefore, $G_{\widehat{B}}$ is contained in $\myacl(BC)$.
	This implies that $G$ is contained in $\myacl(A'C) \cup \myacl(BC)$, which contradicts $A \myinddiv_CB$ by Theorem \ref{thm:dividing}.
	We have shown that $\myacl(A'C) \cup \myacl(\widehat{B}C)$ does not contain exceptional graphs.
	We get ${A'} \myinddiv_C\widehat{B}$ by Theorem \ref{thm:dividing}.
\end{proof}

%\begin{definition}[{\cite{EO}}]
%	A formula $\phi(\overline{x},\overline{a})$ \textit{strongly divides} over $A$ if $\mytp(\overline{a}/A)$ is nonalgebraic and $\{\phi(\overline{x)},\overline{a'}\;|\;\overline{a}'\models \mytp(\overline{a}/A)\}$ is $k$-inconsistent for some $k<\omega$.
%	We say that $\phi(\overline{x},\overline{a})$ \textit{\textarc{\th}-divides} over $A$ if there exists $\overline{c}$ such that $\phi(\overline{x},\overline{a})$  strongly divides over $A\overline{c}$.
%	We define \textarc{\th}-forking of formulas, types and the notion of thorn independence $\myindthorn$ by replacing the words `divide' and `fork' in the definition of forking of formulas, types and the notion of forking independence $\myindfork$ with the words `\textarc{\th}-divide' and `\textarc{\th}-fork' in the same manner as \cite[Section 7.1]{TZ}.
%\end{definition}

\begin{definition}
A theory $T$ admits \textit{elimination of imaginaries} if, for every $\overline{e} \in \mathbb M^{\text{eq}}$, there exists a real tuple $\overline{a}$ such that $\overline{e} \in \mydcleq(\overline{a})$ and $\overline{a} \in \mydcleq(\overline{e})$.
$T$ admits \textit{weak elimination of imaginaries} if, for every $\overline{e} \in \mathbb M^{\text{eq}}$, there exists a real tuple $\overline{a}$ such that $\overline{e} \in \mydcleq(\overline{a})$ and $\overline{a} \in \myacleq(\overline{e})$.
\end{definition}

If $n=2$, $\randfrcolgT$ does not admit elimination of imaginaries, but admits weak elimination imaginaries \cite[Remark 4.26, Theorem 4.29]{CK}.
If $n>2$, even the weaker version does not hold.
\begin{proposition}\label{prop:weak_elim_im}
	If $n>2$, $\randfrcolgT$ does not admit weak elimination of imaginaries.
\end{proposition}
\begin{proof}
	Consider the $\emptyset$-definable equivalence relation $C(x,y)$ defined by $$\mathbb M \models C(x,y) \Leftrightarrow \mycol(x)=\mycol(y).$$
	Let $e$ be the equivalence class of vertices of color $1$.
	Let $\overline{a} \in  \mathbb M$ with ${e} \in \mydcleq(\overline{a})$.
	We show that $\overline{a} \notin \myacleq({e})$.
	Assume for contradiction that such a tuple $\overline{a}$ exists.
	Let $M$ be an arbitrary small model of $\randfrcolgT$ containing $\overline{a}$.
	Put $D:=\{v \in M\;|\; \mycol(v)=1\}$.
	Let $(M_i\;|\;i<\omega)$ be a family of copies of $M$ with $M_0=M$.
	Let $N_0:=M_0$, $N_i:=M_i \otimes_D N_{i-1}$ for $i>0$ and $N:=\bigcup_{i<\omega} N_i$.
	Observe that $D$ is a model of $\frcolgTc$ because $n>2$.
	By Lemma \ref{lem:amalgam_free_colored}, $N_i$ is $\compgraph$-free for $i<\omega$.
	Therefore, $N$ is $\compgraph$-free.
	
	We can embed $N$ into $\mathbb M$ fixing $M$ pointwise by Proposition \ref{prop:fr_monster_basic}(1).
	Assume for contradiction that $\overline{a} \in \myacl^{\text{eq}}({e})$.
	There exists an $(\colorgL)^{\text{eq}}$-formula $\phi(\overline{x},{e})$ such that $\mathbb M \models \phi(\overline{a},{e})$ and $\phi(\overline{x},{e})$ has finitely many solutions.
	Let $\sigma_i:M \to M_i$ be the canonical isomorphism, and put $\overline{a}_i:=\sigma_i(\overline{a})$ for $i<\omega$.
	Observe that $e$ is fixed for every automorphism $\mathbb M \to \mathbb M$.
	%Since $\randfrcolgT$ is model complete by Lemma \ref{lem:fr_model_complete}, $\mathbb M \models  \phi(\overline{a}_i,{e})$ for every $i<\omega$, which contradicts that $\phi(\overline{x},e)$ has finitely many solutions.
	We have $\mathbb M \models  \phi(\overline{a}_i,{e})$ for every $i<\omega$, which contradicts that $\phi(\overline{x},e)$ has finitely many solutions.
\end{proof}


\begin{thebibliography}{99}
\bibitem{Adler}
H.~Adler,
\emph{A geometric introduction to forking and thorn-forking},
J. Math. Logic, \textbf{9}(1) (2009), 1--20.
	
%\bibitem{BS}
%J.~T.~Baldwin and N.~Shi,
%\emph{Stable generic structures},
%Ann.~Pure Appl.~Logic, \textbf{79} (1996), 1--35.

\bibitem{ChanK}
C.~C.~Chang and H.~J.~Keisler,
\emph{Model theory: third edition},
Dover Publications, INC., New York, 2012.

\bibitem{Chernikov}
A.~Chernikov,
\emph{Theories without the tree property of the second kind},
Ann.~Pure Appl.~Logic, \textbf{165}(2) (2014), 695--723.
	
\bibitem{Conant}
G.~Conant, 
\emph{An axiomatic approach to free amalgamation},
J.~Symbolic Logic, \textbf{82}(2) (2017), 648--671.

\bibitem{Conant2}
G.~Conant, 
\emph{Forking and dividing in Henson graphs},
Notre Dame J.~Formal Logic, \textbf{58}(4) (2017), 555--566.

\bibitem{CK}
G.~Conant and A.~Kruckman, 
\emph{Independence in generic incidence structures},
J.~Symbolic Logic, \textbf{84}(2) (2019), 750--780.

\bibitem{E}
C.~d'Elb\'ee,
\emph{Axiomatic theory of independence relations in model theory},
preprint, arXiv:2308.07064 (2023).

%\bibitem{EO}
%C.~Ealy and A.~Onshuus,
%\emph{characterizing rosy theories},
%J.~Symbolic Logic, \textbf{72}(3) (2007), 919--940.

%\bibitem{KR}
%A.~Kruckman and N.~Ramsey,
%\emph{Generic expansion and Skolemization in $\operatorname{NSOP}_1$ theories},
%Ann.~Pure Appl.~Logic, \textbf{169}(8) (2018), 755--774.

\bibitem{KR_Kim}
I.~Kalpen and N.~Ramsey,
\emph{On Kim-independence},
J. Eur.~Math.~Soc., \textbf{22}(5) (2020), 1423--1474.

\bibitem{M}
S.~Mutchnik,
\emph{On $\mynsop_2$ theories},
J. Eur.~Math.~Soc., \textbf{28}(8) (2026), 3475--3498.

\bibitem{Shelah}
S. Shelah,
\emph{Toward classifying unstable theories},
Ann.~Pure Appl.~Logic, \textbf{80} (1996), 229--255.

%\bibitem{Simon}
%P.~Simon, 
%\emph{On dp-minimal ordered structures},
%J.~Symbolic Logic, \textbf{76}(2) (2011), 448--460.

\bibitem{Simon_NIP}
P. Simon,
\emph{A guide to NIP theories},
Lecture Notes in Logic,
Cambridge University Press, Cambridge, 2015.

\bibitem{TZ}
K.~Tent and M.~Ziegler,
{A course in model theory},
Lecture notes in logic, vol. 40.
Cambridge University Press, Cambridge, 2012. 

\end{thebibliography}
\end{document}